\documentclass[onefignum,onetabnum]{siamart251216}

\usepackage{amssymb}
\usepackage{mathtools}
\usepackage{graphicx}

\DeclareMathOperator{\circum}{circ}
\DeclareMathOperator{\aff}{aff}
\DeclareMathOperator{\cone}{cone}
\DeclareMathOperator{\conv}{conv}
\DeclareMathOperator{\closu}{cl}

\DeclareMathOperator{\inte}{int}
\DeclareMathOperator{\dom}{dom}
\DeclareMathOperator{\proj}{P}
\DeclareMathOperator{\Tr}{tr}
\DeclareMathOperator*{\argmin}{arg\,min}
\DeclareMathOperator*{\argmax}{arg\,max}

\IfFileExists{bbm.sty}{\usepackage{bbm}}{\providecommand{\mathbbm}{\mathbb}}
\newcommand{\RR}{\mathbbm{R}}

\newcommand{\Smat}{\mathbb{S}}
\newcommand{\KK}{\mathcal{K}}
\newcommand{\Hcal}{\mathcal{H}}
\newcommand{\Tcone}{\mathcal{T}}
\newcommand{\Kpolar}{\KK^{\circ}}
\newcommand{\norm}[1]{\left\lVert #1 \right\rVert}
\newcommand{\scal}[2]{\left\langle #1, #2 \right\rangle}
\newcommand{\one}{\mathbf{1}}

\usepackage{tikz}
\usetikzlibrary{arrows.meta,calc,patterns}
\usepackage{enumitem}
\usepackage{caption}

\newsiamthm{example}{Example}
\newsiamremark{remark}{Remark}

\headers{Geometry of circumcentric directions of cones}{Y.~Bello-Cruz}

\title{On the geometry of circumcentric directions of cones%
\thanks{\today.
}}

\author{Yunier Bello-Cruz
\thanks{Department of Mathematical Sciences, Northern Illinois University,
DeKalb, IL 60115, USA (\email{yunier.bello@niu.edu},
\url{https://orcid.org/0000-0002-7877-5688}).}
}

\begin{document}
\maketitle

\begin{abstract}
Behling, Bello-Cruz, Lara-Urdaneta, Oviedo, and Santos showed that
the circumcentric direction $d$ of a finitely generated polyhedral
cone $\KK\subset\RR^n$ admits an inscribed Euclidean ball of radius
$\norm{d}^2$ inside the polar cone $\Kpolar$. We sharpen this result
in several ways. The exact set of admissible perturbations is a
polyhedron, strictly larger than the inscribed ball off the
generators and unbounded along $\Kpolar$. From it we read off a
closed form for $\norm{d}^2$ in terms of the inverse Gram matrix of
the conic base, with two-sided spectral bounds, and an aperture
identity $\norm{d}=\cos\theta$ relating the generators to the axis
$-d/\norm{d}$. The inscribed-ball estimate extends to closed convex pointed cones
under one geometric condition: the normalized extremal section
$E_\KK$ has affine hull avoiding the origin. The admissible set is
then the intersection of half-spaces indexed by $E_\KK$, and the
inscribed ball touches its boundary along $\norm{d}^2\,\closu E_\KK$.
A Jordan-frame argument verifies the hypothesis for every simple
symmetric cone and gives $\norm{d}^2=1/r$ for the Jordan rank $r$;
the same value $1/n$ shows up for the doubly nonnegative cone, the
direct-product case obeys the parallel-resistance rule
$1/\norm{d}^2=\sum_\ell 1/\norm{d_\ell}^2$, and the $p$-cones with
$p\ne 2$ provide a clean obstruction. We close with a sharp formula
for the largest step from $d$ along a prescribed direction, worked
out for $L_\infty$-ball constrained least squares and second-order
cone programming; a piecewise smooth version where the inner Slater
condition is exactly Mangasarian--Fromovitz; and a Bregman analogue
covering a Mahalanobis instance and a mirror-descent step.
\end{abstract}



\section{Introduction}\label{sec:intro}

Let $\KK\subset\RR^n$ be a finitely generated polyhedral cone with
normalized conic base $B_\KK=\{u^1,\ldots,u^p\}$, the unit vectors
along the (finitely many) extreme rays of~$\KK$. The
{circumcentric direction} of $B_\KK$ is $d:=-\circum(B_\KK)$, where
the circumcenter is uniquely determined
by~\cite[Prop.~3.3]{BauschkeOuyangWang:2018} as
\begin{align}\label{eq:circ-explicit}
\circum(B_\KK)
= u^1 + \alpha_1 (u^2 - u^1) + \cdots + \alpha_{p-1} (u^p - u^1),
\end{align}
with $(\alpha_1,\ldots,\alpha_{p-1}) \in \RR^{p-1}$ the unique
solution of the $(p-1)\times(p-1)$ linear system whose $i$-th
equation is
\begin{align}\label{eq:circ-system}
\sum_{j=1}^{p-1} \alpha_j
\langle u^{j+1} - u^1,\; u^{i+1} - u^1 \rangle
= \tfrac{1}{2} \|u^{i+1} - u^1\|^2,\qquad i=1,\ldots,p-1.
\end{align}
Behling, Bello-Cruz, Lara-Urdaneta, Oviedo, and
Santos~\cite{BehlingEtAl_CDC} proved that $d$ is caracterized as
\begin{align}\label{eq:dintro}
d=-\proj_{\aff(B_\KK)}(0).
\end{align}
Moreover, the inscribed-ball estimate
\begin{align}\label{eq:behball}
\norm{v}\le \norm{d}^2\implies d+v\in\Kpolar,
\end{align}
where $\Kpolar=\{w:\scal{w}{z}\le 0\text{ for all }z\in\KK\}$ is the
polar cone. The estimate is sharp on the sphere of radius
$\norm{d}^2$: equality holds at every generator $u^i$. The direction
$d$ thus provides a built-in feasible search direction inside
$\Kpolar$, available to projection-based methods and descent schemes
for conic constraints. The circumcenter itself has a longer history in the family of
{\em circumcentered-reflection methods} (CRM), initiated by Behling,
Bello-Cruz, and Santos~\cite{BBS:DR} to accelerate the
Douglas--Rachford method~\cite{BauschkeBelloCruzNghiaPhanWang:2014}
and the alternating-projection method~\cite{BauschkeBelloCruzNghiaPhanWang:2016}
in the affine setting. CRM has been developed extensively in
subsequent work; see, for
instance,~\cite{BBS:linconv,Behling:NA2021,BBS:block,ABBIS:rates,Behling:MP2024,Behling:SIOPT2024,Behling:COA2024,Barros:PPP2025,BC:cCRM2026}.
In all of these schemes the circumcenter is recomputed at each step
from auxiliary point sets that depend on the iterate. Here the
circumcenter is computed once, on the conic base $B_\KK$, and the
resulting $d$ is used as a fixed feasible direction inside
$\Kpolar$, no auxiliary points, no inner iteration.

The inscribed-ball estimate~\eqref{eq:behball} is one consequence of
\cref{eq:circ-system}, not its whole content. The set of $v$ for
which $d+v\in\Kpolar$ is itself a polyhedron,
\begin{align}\label{eq:Pintro}
\mathcal{P}=\bigl\{v\in\RR^n:\max_{1\le i\le p}\scal{v}{u^i}\le\norm{d}^2\bigr\},
\end{align}
and the Euclidean ball $\mathcal{B}=\{\norm{v}\le\norm{d}^2\}$ is
the largest ball it contains, but $\mathcal{P}$ is much larger:
$\mathcal{B}$ touches $\partial\mathcal{P}$ at exactly the points
$u^1,\ldots,u^p$ and sits strictly inside everywhere else, while
$\mathcal{P}$ stretches to infinity along $\Kpolar$. Replacing the
spherical perturbation by a prescribed direction $w$ yields a
closed-form maximum interior step typically much larger than
$\norm{d}^2$, and infinite for $w\in\Kpolar$.

The main contributions of this paper are as follows.

\noindent{\bf (C1)} An exact polyhedral description of $\{v:d+v\in\Kpolar\}$
together with a sharp formula for the directional depth
$\rho_\KK(w)=\sup\{t\ge 0:d+tw\in\Kpolar\}$
(\cref{prop:exact,cor:dirdepth}).

\noindent{\bf (C2)} A closed-form expression for both $d$ and $\norm{d}^2$
in terms of the inverse Gram matrix $M$ of $B_\KK$, with two-sided
spectral bounds in $\lambda_{\min}(M)/p$ and $\lambda_{\max}(M)/p$,
where $p=|B_\KK|$ is the number of generators (\cref{prop:gram}).

\noindent{\bf (C3)} An aperture interpretation: $\norm{d}$ is the cosine
of the common angle between every generator and the axis
$-d/\norm{d}$ (\cref{prop:aperture}). This gives an angular bound
on the directional depth $\rho_\KK(w)$ from~{\bf (C1)} that depends
only on $\norm{d}$ and on the angle of the direction $w$ to the axis
(\cref{cor:angular}).

\noindent{\bf (C4)} An extension of the inscribed-ball estimate to
closed convex pointed cones whose normalized extreme-ray set
$E_\KK$ has affine hull avoiding the origin, that is,
$0\notin\closu\aff(E_\KK)$ (\cref{thm:nonpoly}). The hypothesis
is shown to be equivalent to the existence of a constant-projection
witness (\cref{lem:hyp}) and is strictly stronger than pointedness
(\cref{ex:degen}). The second-order cone and the
positive-semidefinite cone are worked out explicitly
(\cref{ex:soc,ex:psd}), with figures
in~\cref{fig:soc,fig:degen}. This condition is a
strict-feasibility-type counterpart of Slater's condition; for the
broader theory of strict feasibility and facial reduction in conic
optimization, see Drusvyatskiy and Wolkowicz~\cite{DW:2017},
Pataki~\cite{Pataki:2017}, and Roshchina and
Tun\c{c}el~\cite{RT:2019}.

\noindent{\bf (C5)} A sharp polar description of the admissible set:
$\mathcal{P}_\KK=\bigcap_{u\in E_\KK}\{\scal{v}{u}\le\norm{d}^2\}$.
The closed ball $\bar B(0,\norm{d}^2)$ is the largest Euclidean ball
it contains, with contact set
$\partial\bar B(0,\norm{d}^2)\cap\partial\mathcal{P}_\KK=\norm{d}^2\,\closu E_\KK$
(\cref{thm:sharp,cor:dirdepth-np}). This pinpoints the directions
along which the inscribed-ball bound is attained and is new even
in the polyhedral case.

\noindent{\bf (C6)} A unified Jordan-algebraic verification of the
hypothesis $0\notin\closu\aff(E_\KK)$ for every simple symmetric
cone, yielding the closed-form value $\norm{d}^2=1/r$ in the Jordan
rank $r$ (\cref{prop:symmetric}). Three further results test the
reach of the hypothesis beyond the symmetric setting: stability
under direct products with the parallel-resistance formula
$1/\norm{d}^2=\sum_\ell 1/\norm{d_\ell}^2$ (\cref{prop:dirsum}); a
positive verification for the doubly nonnegative cone
$\mathrm{DNN}^n$, again with $\norm{d}^2=1/n$ (\cref{ex:dnn}); and
a clean obstruction at the family of $p$-cones for $p\ne 2$
(\cref{prop:pcone}). The orthant, the second-order cone, and the
positive-semidefinite cone end up on a single footing through the
Pierce decomposition of Faraut and Kor\'anyi~\cite{FK:1994}, and
the geometric content of the hypothesis emerges as planarity of the
extremal section.

\noindent{\bf (C7)} A sharp formula for the largest step along an
arbitrary prescribed direction that keeps $d+\sigma w$ inside
$\Kpolar$ (\cref{prop:step}); this is the step-length oracle of a
feasibility-corrected projected gradient update. A piecewise smooth
version of the active-cone construction
of~\cite[Cor.~2.8]{BehlingEtAl_CDC} is also given
(\cref{cor:piecewise}); the inner Slater condition there is exactly
the Mangasarian--Fromovitz qualification on the active pieces
(\cref{rem:mfcq}). Two concrete problem classes are worked out in
\Cref{sec:concrete}: $L_\infty$-ball constrained least squares
(piecewise smooth, polyhedral feasible set) yields the closed-form
oracle~\cref{eq:linfsigma} via \cref{cor:piecewise,prop:gram},
while second-order cone programming (smooth, non-polyhedral
feasible set) yields the parallel oracle~\cref{eq:socsigma} via
\cref{prop:step}, with explicit two-constraint
formula~\cref{eq:soctwosigma} interpolating between the orthogonal
and co-aligned limits. Both connect to first-order convergence
theory for constrained convex problems; see~\cite{NPR:2019,NNG:2019}.

\noindent{\bf (C8)} A Bregman extension of the construction
(\Cref{sec:bregman}). For a Legendre function $h$ with $\nabla h(0)=0$,
the Bregman projection $c_h$ of the origin onto $\aff(B_\KK)$
produces a dual identity $\scal{\nabla h(c_h)}{u^i}=\kappa_h$
uniformly in $i$ (\cref{lem:bregman-key}), from which a Bregman
inscribed-ball estimate $\norm{v}\le\kappa_h\Rightarrow d_h+v\in\Kpolar$
(\cref{thm:bregman-ball}) follows by the same Cauchy--Schwarz
argument. Beyond the rescaling family $h_p(x)=\tfrac1p\norm{x}^p$,
the Mahalanobis quadratic $h(x)=\tfrac12 x^\top A x$
(\cref{ex:maha}) gives a genuinely new direction $d_h\ne d$ with
margin $\kappa_h$ that may exceed the Euclidean $\norm{d}^2$. The
construction yields a Bregman feasibility-corrected
step~\cref{eq:bcfg} parallel to \cref{prop:step}, and is compared
to the Ouyang--Wang Bregman circumcenter of finitely many
points~\cite{Ouyang:2021,Ouyang:2023} in \cref{rem:ouyang}: the
two notions coincide in the Euclidean case when the point set is
$B_\KK$ but disagree under any non-Euclidean $h$, since
equidistance among finitely many points and single-projection
feasibility certification answer different questions.

The paper is organized as follows. \Cref{sec:setup} fixes notation,
collects the auxiliary results invoked in the proofs, and recalls
the basic identity. \Cref{sec:exact,sec:gram,sec:aperture} concern
the polyhedral case. \Cref{sec:nonpoly} extends the construction
beyond polyhedrality, proves the sharpness theorem, treats the
canonical symmetric-cone instances, and exhibits a pointed
counterexample. \Cref{sec:step,sec:piecewise} return to algorithmic
applications. \Cref{sec:concrete} works out two concrete problem
classes, $L_\infty$-ball constrained least squares and second-order
cone programming, with closed-form step-length oracles.
\Cref{sec:bregman} develops the Bregman analogue. \Cref{sec:outlook}
closes with directions for future work.

\section{Preliminaries and the basic identity}\label{sec:setup}

\subsection{Auxiliary results from the literature}\label{sec:prelim}

The proofs draw on a few classical results, which we collect here
for ease of reference. The Euclidean argument uses Cauchy--Schwarz
and orthogonal projection onto closed affine subspaces. Pointedness
of $\KK$ is equivalent, by Hahn--Banach, to the existence of
$\mu\in\RR^n$ with $\scal{x}{\mu}>0$ on $\KK\setminus\{0\}$, and the
resulting compact convex base $C=\{x\in\KK:\scal{x}{\mu}=1\}$
together with Krein--Milman gives the conic-hull representation
$\KK=\closu\cone(E_\KK)$~\cite{Rockafellar:1970,BC:book}; the
Hilbert-space version mentioned in \cref{rem:hilbert} replaces
Krein--Milman with the Choquet integral
representation~\cite{Phelps:2001}. The spectral bounds in
\cref{prop:gram} rest on the Rayleigh-quotient inequalities and
Weyl's monotonicity theorem~\cite{HornJohnson:2013}. The unified
treatment of symmetric cones in \cref{prop:symmetric} uses the
Pierce decomposition, the Jordan spectral theorem, and the
identification of extreme rays with primitive
idempotents~\cite{FK:1994}. The algorithmic sections invoke Slater's condition (which yields
$\Kpolar=\Tcone_\Omega(\bar x)$ through~\cite[Cor.~2.8]{BehlingEtAl_CDC},
where $\Tcone_\Omega(\bar x)=\{v\in\RR^n: \exists v_k\to v,\ t_k\downarrow 0,\ \bar x+t_kv_k\in\Omega\}$
is the (Bouligand) tangent cone), the Mangasarian--Fromovitz
constraint qualification~\cite{Mangasarian:1967}, and the convex
subdifferential of a maximum of finitely many smooth convex
functions,
$\partial(\max_i g_i)(\bar x)=\conv\{\nabla g_i(\bar x):g_i(\bar x)=\max_j g_j(\bar x)\}$~\cite{Clarke:1983}.
The Bregman material in~\Cref{sec:bregman} uses Legendre functions
in the sense of Rockafellar (proper, lower semicontinuous, convex,
essentially smooth and essentially strictly convex; see
\Cref{sec:bregman} for the precise conditions); the only
nonstandard fact invoked is strict monotonicity of $\nabla h$ on
$\inte\dom h$, that is, $\scal{\nabla h(x)-\nabla h(y)}{x-y}>0$
for distinct $x,y\in\inte\dom h$, which follows from essential
strict convexity~\cite{Rockafellar:1970,BC:book}.

\subsection{Setting and the basic identity}\label{sec:basic}

We follow the conventions of~\cite{BehlingEtAl_CDC}.

Throughout,
$\KK\subset\RR^n$ is a polyhedral cone with normalized
conic base $B_\KK=\{u^1,\ldots,u^p\}$, the $u^i$ being unit vectors
along the extreme rays of~$\KK$. The polar cone is
\begin{align}
\Kpolar=\{w\in\RR^n:\scal{w}{z}\le 0,\quad \forall z\in \KK\}=\{w\in\RR^n: \scal{w}{u^i}\le 0,\text{ }i=1,\ldots,p\},
\end{align}
and the circumcentric direction is, as in~\cref{eq:dintro},
\begin{align}\label{eq:circd}
d:=-\circum(B_\KK)=-\proj_{\aff(B_\KK)}(0).
\end{align}
Since $\circum(B_\KK)\in\aff(B_\KK)$ and $-d=\proj_{\aff(B_\KK)}(0)$
is orthogonal to the linear part of $\aff(B_\KK)$, the identity
$\scal{u^i-(-d)}{-d}=0$ holds for every $i$, equivalently
\cite[Lemma~2.5(ii)]{BehlingEtAl_CDC},
\begin{align}\label{eq:keyident}
\scal{d}{u^i}=-\norm{d}^2,\qquad i=1,\ldots,p.
\end{align}
Every generator has the same inner product $-\norm{d}^2$ with $d$.
The inscribed-ball theorem~\cref{eq:behball}
of~\cite[Thm.~2.6]{BehlingEtAl_CDC} follows in one line
from~\cref{eq:keyident} by Cauchy--Schwarz.

\section{The exact admissible set is a polyhedron}\label{sec:exact}

The inclusion $d+v\in\Kpolar$ amounts to $\scal{d+v}{u^i}\le 0$ for
every $i$, and $\scal{d}{u^i}$ is the constant $-\norm{d}^2$
by~\cref{eq:keyident}. Reading the resulting linear inequalities
off explicitly gives the exact admissible set.

\begin{proposition}[exact admissible polyhedron]\label{prop:exact}
For every $v\in\RR^n$,
\begin{align}
\label{eq:exact}
d+v\in\Kpolar &\iff \max_{1\le i\le p}\scal{v}{u^i}\le \norm{d}^2,\\
\label{eq:exact-int}
d+v\in\inte\Kpolar &\iff \max_{1\le i\le p}\scal{v}{u^i}<\norm{d}^2.
\end{align}
The polyhedron
$\mathcal{P}:=\{v\in\RR^n:\max_i\scal{v}{u^i}\le \norm{d}^2\}$ has
inscribed Euclidean ball, centered at the origin, of radius
exactly~$\norm{d}^2$.
\end{proposition}

\begin{proof}
Since $\KK=\cone(B_\KK)$, $d+v\in\Kpolar$ is the conjunction of
$\scal{d+v}{u^i}\le 0$ over $i=1,\ldots,p$, which by~\cref{eq:keyident}
is~\cref{eq:exact}. The interior version is the strict-inequality
counterpart, using
$\inte\Kpolar=\{w:\scal{w}{u^i}<0\text{ for all }i\}$. The
inscribed-ball radius of~$\mathcal P$ is
$\norm{d}^2/\max_i\norm{u^i}=\norm{d}^2$.
\end{proof}

The polyhedron $\mathcal{P}$ improves on the inscribed-ball
estimate in two ways. Outside the (finitely many) directions
$u^1,\ldots,u^p$, the inclusion
$\{\norm{v}\le\norm{d}^2\}\subsetneq\mathcal{P}$ is strict: along
every other direction, perturbations of norm strictly greater than
$\norm{d}^2$ are admissible. The recession cone of~$\mathcal{P}$
equals~$\Kpolar$, so any ball of finite radius misses an entire
cone of admissible perturbations. The next corollary makes both
observations quantitative.

\begin{corollary}[directional depth in $\Kpolar$]\label{cor:dirdepth}
For $w\in\RR^n\setminus\{0\}$, set
$\rho_\KK(w):=\sup\{t\ge 0:d+tw\in\Kpolar\}$. Then
\begin{align}\label{eq:rho}
\rho_\KK(w)=
\begin{cases}
+\infty, & \max_i\scal{w}{u^i}\le 0,\\[1mm]
\dfrac{\norm{d}^2}{\max_i\scal{w}{u^i}}, & \text{otherwise.}
\end{cases}
\end{align}
The infimum of~$\rho_\KK$ over the unit sphere equals $\norm{d}^2$
and is attained exactly at $w\in B_\KK$.
\end{corollary}

\begin{proof}
By~\cref{prop:exact}, $d+tw\in\Kpolar$ iff
$t\scal{w}{u^i}\le\norm{d}^2$ for every~$i$. If
$\max_i\scal{w}{u^i}\le 0$ every $t\ge 0$ qualifies; otherwise the
binding constraint sits at any
$i^*\in\argmax_i\scal{w}{u^i}>0$. Over the unit sphere,
$\sup_w\max_i\scal{w}{u^i}=\max_i\norm{u^i}=1$, attained at $w=u^i$,
so the infimum of $\rho_\KK$ is $\norm{d}^2$.
\end{proof}

Two special cases stand out. For $w\in B_\KK$, identity
\cref{eq:rho} returns $\rho_\KK(w)=\norm{d}^2$, the inscribed-ball
regime. For $w\in\Kpolar$, $\rho_\KK(w)=+\infty$, since
$-d\in\Kpolar$ keeps $d+tw\in\Kpolar$ for every $t\ge 0$. The
example below puts numbers on both extremes and an interior case
in $\RR^2$, and quantifies the gap between \cref{prop:exact} and
the inscribed-ball estimate.

\begin{example}[the nonnegative orthant]\label{ex:orthant}
Take $\KK=\RR^2_{+}$, so $B_\KK=\{e_1,e_2\}$ and
$\Kpolar=-\RR^2_{+}$. A direct computation gives
\begin{align}
d=-\tfrac12(e_1+e_2),\qquad \norm{d}^2=\tfrac12.
\end{align}
\Cref{prop:exact} returns the unbounded polyhedron
\begin{align}
\mathcal{P}=\bigl\{v\in\RR^2:v_1\le \tfrac12,\ v_2\le \tfrac12\bigr\},
\end{align}
whereas Theorem~2.6 of~\cite{BehlingEtAl_CDC} only guarantees the
disc $\mathcal{B}=\{\norm{v}\le \tfrac12\}$. The disc is tangent
to $\partial\mathcal P$ at the two points $(\tfrac12,0)$ and
$(0,\tfrac12)$, that is, exactly along the generators $e_1,e_2$,
and strictly inside $\mathcal P$ everywhere else.
\Cref{fig:orthant} displays both sets and the gap. Two concrete
witnesses of the gap are immediate:
\begin{itemize}[leftmargin=*]
\item Every $v=-t(e_1+e_2)$ with $t\ge 0$ belongs to~$\mathcal P$,
yet $\norm{v}=t\sqrt 2$ is unbounded; the half-line sits inside
the recession cone of~$\mathcal P$ and is missed by any ball.
\item Along the diagonal $w=(e_1+e_2)/\sqrt 2$,
identity~\cref{eq:rho} gives $\rho_\KK(w)=\norm{d}^2/(1/\sqrt 2)=1/\sqrt 2$,
strictly larger than the inscribed-ball value~$\tfrac12$.
\end{itemize}
\end{example}

\begin{figure}[ht]
    \centering
    \includegraphics[width=0.6\textwidth]{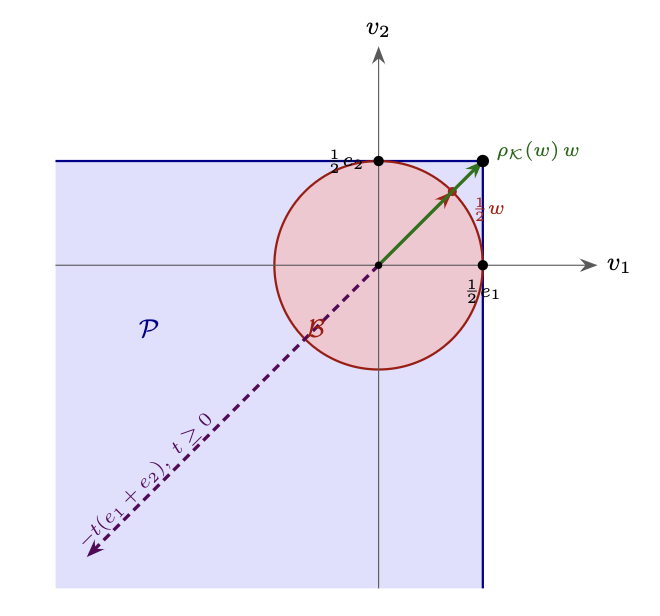}
    \caption{The orthant case of \cref{ex:orthant}. The exact admissible
set $\mathcal{P}=\{v_1\le\tfrac12,\ v_2\le\tfrac12\}$ from
\cref{prop:exact} is the unbounded $L$-shaped region (blue); the
inscribed-ball estimate $\mathcal{B}=\{\norm{v}\le\tfrac12\}$ of
\cite{BehlingEtAl_CDC} is the disc (red), tangent to
$\partial\mathcal{P}$ only at $\tfrac12 e_1$ and $\tfrac12 e_2$.
Along the diagonal $w=(e_1+e_2)/\sqrt 2$, the inscribed ball reaches
$\tfrac12 w$ (short red arrow) while \cref{cor:dirdepth} delivers
$\rho_\KK(w)=1/\sqrt 2$, taking $0$ to the corner of~$\mathcal{P}$
(green arrow). The recession half-line $-t(e_1+e_2)$ lies entirely
in~$\mathcal{P}$ and is missed by every ball.}
\label{fig:orthant}
\end{figure}
\section{Closed form via the Gram matrix}\label{sec:gram}

\Cref{prop:exact} described $\mathcal{P}$ in terms of the unit
generators $u^1,\ldots,u^p$ and the scalar $\norm{d}^2$. The
generators are typically given as inputs; $\norm{d}^2$ and~$d$
itself are not. The next result reduces their computation to a
single $p\times p$ Gram-matrix inversion and provides spectral
bounds that quantify how close $B_\KK$ is to losing linear
independence.

\begin{proposition}[Gram-matrix formula]\label{prop:gram}
Suppose the base $B_\KK=\{u^1,\ldots,u^p\}$ of the cone $\KK$ is
linearly independent, and let $M\in\RR^{p\times p}$,
$M_{ij}=\scal{u^i}{u^j}$, be its Gram matrix. Then
\begin{align}\label{eq:gram-d}
d=-\frac{1}{\one^\top M^{-1}\one}\sum_{i=1}^{p}(M^{-1}\one)_{i}\,u^i,\qquad
\norm{d}^2=\frac{1}{\one^\top M^{-1}\one},
\end{align}
and
\begin{align}\label{eq:gram-bds}
\frac{\lambda_{\min}(M)}{p}\;\le\;\norm{d}^2\;\le\;\frac{\lambda_{\max}(M)}{p}\;\le\;1.
\end{align}
\end{proposition}

\begin{proof}
Writing $\circum(B_\KK)=\sum_i\lambda_i u^i$ with
$\one^\top\lambda=1$, the minimum-norm characterization of the
circumcenter \cite[Lemma~2.5(i)]{BehlingEtAl_CDC} says that
$\lambda$ minimizes $\lambda^\top M\lambda$ subject to
$\one^\top\lambda=1$. The Lagrange condition $2M\lambda=\mu\one$
together with $\one^\top\lambda=1$ gives
$\lambda=M^{-1}\one/(\one^\top M^{-1}\one)$, and a direct
substitution yields
$\norm{\circum(B_\KK)}^2=\lambda^\top M\lambda=1/(\one^\top M^{-1}\one)$.
This proves~\cref{eq:gram-d}. The Rayleigh-quotient
inequalities~\cite{HornJohnson:2013} give
$$p/\lambda_{\max}(M)\le\one^\top M^{-1}\one\le p/\lambda_{\min}(M),$$
which yields, on reciprocation, the spectral bounds
in~\cref{eq:gram-bds}; the rightmost inequality is
$\lambda_{\max}(M)\le \Tr M=p$.
\end{proof}

The two-sided bound~\cref{eq:gram-bds} places $\norm{d}^2$ between
$\lambda_{\min}(M)/p$ and $\lambda_{\max}(M)/p$. The lower bound is
the quantitative form of the qualitative loss-of-pointedness
phenomenon in~\cite[Prop.~3.2]{BehlingEtAl_CDC}.

\begin{remark}[loss of pointedness]\label{rem:cond}
The lower bound $\lambda_{\min}(M)/p$ in~\cref{eq:gram-bds} measures
quantitatively how far the conic base is from positive linear
dependence: it vanishes precisely as $B_\KK$ approaches dependence,
in agreement with Proposition~3.2 of~\cite{BehlingEtAl_CDC}. The
upper bound is saturated when $B_\KK$ is orthonormal; in that case
$M=I_p$ and $d=-(1/p)\sum_i u^i$, $\norm{d}^2=1/p$.
\end{remark}

A second consequence of the closed form~\cref{eq:gram-d} is
quantitative stability of $d$ under perturbations of the
generators. This is what one needs when the active set changes
between iterations of an algorithm and the conic base of the new
active cone differs only slightly from the previous one.

\begin{remark}[continuity in the generators]\label{rem:cont}
Equation~\cref{eq:gram-d} expresses~$d$ as a rational function of
the entries of~$M$, which is itself smooth in the generators. On
any open neighborhood where $\{u^1,\ldots,u^p\}$ remains linearly
independent, $(u^1,\ldots,u^p)\mapsto d$ is therefore real-analytic,
with local Lipschitz constant controlled by~$\norm{M^{-1}}$. This
is the quantitative form of the stability needed when~$d$ is
updated between iterations.
\end{remark}

\section{An aperture interpretation}\label{sec:aperture}

Identity~\cref{eq:keyident} also has a geometric reading. All
generators have the same inner product with $-d$, so they make the
same angle with the unit vector along $-d$. The conic base lies on
a single right circular cone of half-aperture $\arccos\norm{d}$ and
axis $-d/\norm{d}$. We write $S^{n-1}=\{x\in\RR^n:\norm{x}=1\}$
for the unit sphere in $\RR^n$.

\begin{proposition}[common aperture of the conic base]\label{prop:aperture}
Suppose $d\ne 0$ and set $a:=-d/\norm{d}\in S^{n-1}$ and
$\theta(B_\KK):=\arccos\norm{d}\in[0,\pi/2]$. Then
$\scal{u^i}{a}=\cos\theta(B_\KK)$ for every $i=1,\ldots,p$, so
$B_\KK$ lies on the bounding circle of the spherical cap
$\{u\in S^{n-1}:\scal{u}{a}\ge\cos\theta(B_\KK)\}$.
\end{proposition}

\begin{proof}
By~\cref{eq:keyident},
$\scal{u^i}{a}=-\scal{u^i}{d}/\norm{d}=\norm{d}^2/\norm{d}=\norm{d}=\cos\theta(B_\KK)$.
\end{proof}

Two limiting regimes explain the geometry. If $B_\KK=\{u\}$ is a
single ray, then $\theta(B_\KK)=0$, $\norm{d}=1$, and the polar
cone is a half-space; the inscribed-ball estimate has its largest
radius. As $B_\KK$ approaches positive linear dependence,
equivalently $\circum(B_\KK)\to 0$, the half-aperture
$\theta(B_\KK)\to\pi/2$, the generators flatten toward a hyperplane
through the origin, and $\norm{d}\to 0$. This is the same
loss-of-pointedness regime quantified in~\cref{rem:cond}. The aperture identity also yields an estimate for the directional
depth~\cref{eq:rho} that uses only two angular quantities: the
half-aperture $\theta$ and the angle between $w$ and the axis
$a=-d/\norm{d}$.

\begin{corollary}[angular bound on directional depth]\label{cor:angular}
Under the assumptions of~\cref{prop:aperture}, fix $w\in S^{n-1}$
and write $\phi:=\arccos\scal{w}{a}\in[0,\pi]$,
$\theta:=\theta(B_\KK)$. Then
\begin{align}\label{eq:angular}
\max_{1\le i\le p}\scal{w}{u^i}\le \cos(\phi-\theta),
\end{align}
and consequently
\begin{align}\label{eq:rhobound}
\rho_\KK(w)\ge \frac{\cos^2\theta}{[\cos(\phi-\theta)]_+},
\end{align}
with $[t]_+=\max\{t,0\}$ and $\rho_\KK(w)=+\infty$ when
$\phi-\theta\ge\pi/2$.
\end{corollary}

\begin{proof}
By~\cref{prop:aperture}, each $u^i$ admits a decomposition
$u^i=\cos\theta\,a+\sin\theta\,e^i$ with unit $e^i\perp a$; write
$w=\cos\phi\,a+\sin\phi\,f$ similarly with unit $f\perp a$ and
$\sin\phi\ge 0$. Then
\begin{align}
\scal{w}{u^i}&=\cos\theta\cos\phi+\sin\theta\sin\phi\,\scal{f}{e^i}\\
&\le\cos\theta\cos\phi+\sin\theta\sin\phi=\cos(\phi-\theta),
\end{align}
which is~\cref{eq:angular}. Substituting into~\cref{eq:rho}
gives~\cref{eq:rhobound}; when $\phi-\theta\ge\pi/2$ the right-hand
side of~\cref{eq:angular} is non-positive and~\cref{eq:rho} returns
$+\infty$.
\end{proof}

The bound~\cref{eq:rhobound} requires only the two scalars
$\norm{d}$ and $\phi$, with no dependence on the individual
generators. This is the right object when $d$ is computed once at
$\bar x$ and reused for many candidate directions $w$, as in the
algorithmic applications of~\Cref{sec:step}.

\section{Beyond polyhedral cones: a sharp polar description}\label{sec:nonpoly}

The constructions of \Cref{sec:exact,sec:gram,sec:aperture} all
rest on a finite conic base. Two of the most prominent cones in
conic optimization fall outside that hypothesis: the second-order
cone $\mathcal{L}^n$ and the positive-semidefinite cone
$\Smat^n_+$ are not polyhedral, and their extreme-ray sets are
continuous manifolds. We isolate here the geometric condition
under which the inscribed-ball theorem extends, and prove a
non-polyhedral counterpart that goes further than its polyhedral
predecessor: it identifies the admissible set as the polar-style
intersection
\begin{align*}
\mathcal{P}_\KK=\bigcap_{u\in E_\KK}\bigl\{v\in\RR^n:\scal{v}{u}\le\norm{d}^2\bigr\},
\end{align*}
and pins down the contact set between the inscribed ball and the
boundary of $\mathcal{P}_\KK$ as the homothetic image
$\norm{d}^2\,\closu E_\KK$. The geometric condition is then
verified for several classes of cones beyond the symmetric ones,
and explicit obstructions are exhibited in two contrasting
situations: a pointed polyhedral cone in $\RR^3$ and the family of
$p$-cones with $p\ne 2$.

\begin{definition}[extremal section]\label{def:Ek}
For a closed convex cone $\KK\subset\RR^n$ define the extremal
section
\begin{align}
E_\KK:=\bigl\{u\in\RR^n:\norm{u}=1,\ \RR_+u\text{ is an extreme ray of }\KK\bigr\}.
\end{align}
\end{definition}

For polyhedral~$\KK$ one has $E_\KK=B_\KK$. In the non-polyhedral
cases of interest below, $E_\KK$ may be infinite or even a
continuous manifold: the second-order cone has $E_\KK$
homeomorphic to a sphere, and the positive-semidefinite cone has
$E_\KK$ homeomorphic to a real projective space.

We work under the hypothesis that $\aff(E_\KK)$ avoids the origin.
The next lemma rephrases this as the existence of a constant on
$E_\KK$, the non-polyhedral analogue of identity~\cref{eq:keyident}.

\begin{lemma}[constant-projection characterization]\label{lem:hyp}
Let $\KK\subset\RR^n$ be a closed convex cone with
$E_\KK\ne\emptyset$. The following are equivalent:
\begin{itemize}
\item[ {\bf (i)}] $0\notin\closu\aff(E_\KK)$;
\item[ {\bf (ii)}] there exist $\nu\in\RR^n\setminus\{0\}$ and
$c\ne 0$ such that
\begin{align}\label{eq:lemhyp}
\scal{u}{\nu}=c\qquad\text{for every }u\in E_\KK.
\end{align}
\end{itemize}
The canonical witness is $(\nu,c)=(p_0,\norm{p_0}^2)$ with
$p_0:=\proj_{\closu\aff(E_\KK)}(0)$, in which case
\cref{eq:lemhyp} reproduces identity~\cref{eq:keyident} for every
$u\in E_\KK$ with $d=-p_0$.
\end{lemma}

\begin{proof}
(i)$\Rightarrow$(ii). Set $A:=\closu\aff(E_\KK)$ and
$p_0:=\proj_A(0)$, which is nonzero by~(i). Orthogonality of the
projection onto a closed affine subspace yields
$\scal{x-p_0}{p_0}=0$ for every $x\in A$, hence
$\scal{x}{p_0}=\norm{p_0}^2$ on $A$, and {a fortiori} on
$E_\KK\subset A$. Take $(\nu,c)=(p_0,\norm{p_0}^2)$.

(ii)$\Rightarrow$(i). The hypothesis says
$E_\KK\subset\{x\in\RR^n:\scal{x}{\nu}=c\}$, hence
$\closu\aff(E_\KK)\subset\{x:\scal{x}{\nu}=c\}$, and
$\scal{0}{\nu}=0\ne c$ excludes the origin.
\end{proof}

The canonical witness $(\nu,c)=(p_0,\norm{p_0}^2)$ produced in the
proof yields $-d=p_0$, the natural extension of~\cref{eq:circd} to
the non-polyhedral setting. The hypothesis is strictly stronger
than pointedness: pointedness of $\KK$ is equivalent, by
Hahn--Banach, to the existence of $\mu\in\RR^n\setminus\{0\}$ with
$\scal{u}{\mu}>0$ {uniformly} for $u\in E_\KK$, whereas the lemma
demands $\scal{u}{\nu}$ {constant} on $E_\KK$, not merely uniformly
positive. \Cref{ex:degen} below exhibits a pointed cone for which
this gap is non-empty.

\begin{theorem}[inscribed-ball estimate beyond polyhedrality]\label{thm:nonpoly}
Let $\KK\subset\RR^n$ be a nontrivial closed convex pointed cone,
and suppose $0\notin\closu\aff(E_\KK)$. Set
\begin{align}\label{eq:dnonpoly}
d:=-\proj_{\closu\aff(E_\KK)}(0).
\end{align}
Then $d\ne 0$, $\scal{d}{u}=-\norm{d}^2$ for every $u\in E_\KK$,
and $d+v\in\Kpolar$ for every $\norm{v}\le \norm{d}^2$.
\end{theorem}

\begin{proof}
The hypothesis gives $d\ne 0$ at once, and \cref{lem:hyp} applied
with $\nu=-d$, $c=\norm{d}^2$ yields $\scal{u}{-d}=\norm{d}^2$,
equivalently $\scal{d}{u}=-\norm{d}^2$, for every $u\in E_\KK$.
For $\norm{v}\le\norm{d}^2$ and $u\in E_\KK$, Cauchy--Schwarz gives
\begin{align}
\scal{d+v}{u}=-\norm{d}^2+\scal{v}{u}\le-\norm{d}^2+\norm{v}\le 0.
\end{align}
Now let $z\in \KK$. Pointedness of $\KK$ yields, via
Hahn--Banach~\cite{Rockafellar:1970}, a vector $\mu\in\RR^n$ with
$\scal{x}{\mu}>0$ on $\KK\setminus\{0\}$, so the cross-section
$C:=\{x\in \KK:\scal{x}{\mu}=1\}$ is a compact convex base of
$\KK$, and Krein--Milman~\cite{BC:book} applied to $C$ gives
$\KK=\closu\cone(E_\KK)$. Pick
$z_k=\sum_{i=1}^{q_k}t_k^i u_k^i\to z$ with $t_k^i\ge 0$ and
$u_k^i\in E_\KK$. Then for each~$k$,
\begin{align}
\scal{d+v}{z_k}=\sum_{i=1}^{q_k}t_k^i\scal{d+v}{u_k^i}\le 0,
\end{align}
and continuity of the inner product yields $\scal{d+v}{z}\le 0$.
\end{proof}

The proof actually delivers more:
$\scal{d}{u}=-\norm{d}^2$ uniformly for $u\in E_\KK$, the
non-polyhedral analogue of~\cref{eq:keyident}. \Cref{prop:exact}
therefore extends verbatim, with the maximum over the finite base
$B_\KK$ replaced by the supremum over $E_\KK$.

The next theorem promotes \cref{thm:nonpoly} from an inscribed-ball
estimate into a sharp polar-style description of the admissible set,
together with an explicit identification of the contact set. Both
statements are new even in the polyhedral case: $E_\KK$ is then
$B_\KK$, the polar description recovers~\cref{eq:exact}, and the
contact set specializes to the finite collection
$\{\norm{d}^2 u^i\}_{i=1}^p$ of \cref{ex:orthant}.

\begin{theorem}[sharpness of the inscribed ball]\label{thm:sharp}
Under the hypotheses of \cref{thm:nonpoly}, the admissible set
\begin{align}\label{eq:Pnp}
\mathcal{P}_\KK:=\{v\in\RR^n:d+v\in\Kpolar\}
\end{align}
admits the polar-style description
\begin{align}\label{eq:Pnp-eq}
\mathcal{P}_\KK=\bigcap_{u\in E_\KK}\{v\in\RR^n:\scal{v}{u}\le\norm{d}^2\}.
\end{align}
The closed Euclidean ball $\bar{B}(0,\norm{d}^2)$ is the largest
ball centered at the origin contained in $\mathcal{P}_\KK$: for
every $r>\norm{d}^2$ and every $u\in E_\KK$, the point $ru$
belongs to $\bar{B}(0,r)\setminus\mathcal{P}_\KK$. Moreover, the
boundary contact set
\begin{align}\label{eq:contact}
\partial\bar{B}(0,\norm{d}^2)\cap\partial\mathcal{P}_\KK
=\norm{d}^2\,\closu E_\KK
\end{align}
is the homothetic image of $\closu E_\KK$ at scale $\norm{d}^2$.
\end{theorem}

\begin{proof}
The proof of \cref{thm:nonpoly} establishes
$\KK=\closu\cone(E_\KK)$. By bilinearity of the inner product and
continuity, the inclusion $d+v\in\Kpolar$, that is,
$\scal{d+v}{z}\le 0$ for every $z\in \KK$, is equivalent to
$\scal{d+v}{u}\le 0$ for every $u\in E_\KK$. Substituting
$\scal{d}{u}=-\norm{d}^2$ from \cref{thm:nonpoly} gives
$\scal{v}{u}\le\norm{d}^2$, which is~\cref{eq:Pnp-eq}.

Containment $\bar{B}(0,\norm{d}^2)\subset\mathcal{P}_\KK$ is
\cref{thm:nonpoly} itself. For maximality, fix $u\in E_\KK$ and
$r>\norm{d}^2$. The point $v=ru$ has $\norm{v}=r$ and
$\scal{v}{u}=r>\norm{d}^2$, so $v\notin\mathcal{P}_\KK$
by~\cref{eq:Pnp-eq}.

Let $u\in\closu E_\KK$ and pick $u_k\in E_\KK$ with $u_k\to u$.
Each $v_k:=\norm{d}^2 u_k$ satisfies $\norm{v_k}=\norm{d}^2$ (since
$\norm{u_k}=1$) and $\scal{v_k}{u_k}=\norm{d}^2\norm{u_k}^2=\norm{d}^2$,
so $v_k\in\partial\bar{B}(0,\norm{d}^2)\cap\partial\mathcal{P}_\KK$.
The limit $v:=\norm{d}^2 u$ has $\norm{v}=\norm{d}^2$ by continuity
of the norm, lies in $\partial\mathcal{P}_\KK$ since
$\partial\mathcal{P}_\KK$ is closed, and is therefore in the
contact set on the left of~\cref{eq:contact}.

Let $v\in\partial\bar{B}(0,\norm{d}^2)\cap\partial\mathcal{P}_\KK$.
The boundary description $\partial\mathcal{P}_\KK$ together
with~\cref{eq:Pnp-eq} gives a sequence $u_k\in E_\KK$ with
$\scal{v}{u_k}\to\norm{d}^2$. Cauchy--Schwarz gives
$\scal{v}{u_k}\le\norm{v}\norm{u_k}=\norm{d}^2$, so the inner
products approach the Cauchy--Schwarz upper bound:
\begin{align}
\biggl\|\frac{v}{\norm{v}}-u_k\biggr\|^2
=2-2\biggl\langle\frac{v}{\norm{v}},u_k\biggr\rangle
=2-\frac{2}{\norm{d}^2}\scal{v}{u_k}\longrightarrow 0.
\end{align}
Hence $u_k\to v/\norm{d}^2$, which therefore lies in
$\closu E_\KK$, and $v=\norm{d}^2\cdot(v/\norm{d}^2)\in\norm{d}^2\,\closu E_\KK$.
\end{proof}

Geometrically, the inscribed ball touches the boundary of the
admissible set along a copy of $E_\KK$ scaled by $\norm{d}^2$. For
polyhedral $\KK$ the contact set is the finite set
$\{\norm{d}^2 u^i\}_{i=1}^p$. For $\mathcal{L}^n$ it is an
$(n-2)$-dimensional sphere on the boundary of $-\mathcal{L}^n$
(visible in~\cref{fig:soc} as the dashed circle when $n=3$), and
for $\Smat^n_+$ it is the rank-one trace-$1/n$ slice
$\{vv^\top/n:\norm{v}=1\}\subset-\Smat^n_+$.

The polar description~\cref{eq:Pnp-eq} immediately yields a
non-polyhedral version of \cref{cor:dirdepth}.

\begin{corollary}[directional depth in the non-polyhedral case]\label{cor:dirdepth-np}
Under the hypotheses of \cref{thm:nonpoly}, set
$\rho_\KK(w):=\sup\{t\ge 0:d+tw\in\Kpolar\}$ for
$w\in\RR^n\setminus\{0\}$. Then
\begin{align}\label{eq:rho-np}
\rho_\KK(w)=
\begin{cases}
+\infty, & \sup_{u\in E_\KK}\scal{w}{u}\le 0,\\[1mm]
\dfrac{\norm{d}^2}{\sup_{u\in E_\KK}\scal{w}{u}}, & \text{otherwise.}
\end{cases}
\end{align}
The infimum of $\rho_\KK$ over the unit sphere equals $\norm{d}^2$
and is attained on $\closu E_\KK$.
\end{corollary}

\begin{proof}
By~\cref{eq:Pnp-eq}, $d+tw\in\Kpolar$ iff
$t\scal{w}{u}\le\norm{d}^2$ for every $u\in E_\KK$. If
$$\sup_{u\in E_\KK}\scal{w}{u}\le 0$$ then every $t\ge 0$ qualifies.
Otherwise the binding constraint sits at the supremum (a maximum
on $\closu E_\KK$ when this set is compact). Over the unit sphere,
$$\sup_w\sup_{u\in E_\KK}\scal{w}{u}=1,$$ attained when
$w\in\closu E_\KK$, so the infimum of $\rho_\KK$ is $\norm{d}^2$.
\end{proof}

The two extremes have a structural reading. Directions $w$ along
$\closu E_\KK$ are the worst extreme rays of $-\KK$ as seen from
$d$, and they realize the inscribed-ball radius $\norm{d}^2$;
directions $w\in\Kpolar$ are insensitive to the perturbation,
giving infinite depth.

\begin{remark}[the polar description as a generalized polyhedron]
\label{rem:Pnp-poly}
Equation~\cref{eq:Pnp-eq} writes $\mathcal{P}_\KK$ as an
intersection of half-spaces indexed by $E_\KK$. When $E_\KK$ is
finite this is a polyhedron in the usual sense; for
$\mathcal{L}^n$ it is the intersection of a continuous family of
tangent half-spaces, recovering the cone
$\{(v_x,v_t):v_t+\norm{v_x}\le\norm{d}^2\}$ from the direct
verification in \cref{ex:soc}; for $\Smat^n_+$ it is the
spectrahedron-like set
$\{V\in\Smat^n:\Tr(VW)\le\norm{d}^2\norm{W}_F\text{ for all rank-one }W\succeq 0\}$,
which simplifies to the Frobenius ball-plus-trace-cone description
underlying \cref{ex:psd}.
\end{remark}

The two canonical instances follow. Both have continuous
extreme-ray sets and both satisfy the hypothesis
of~\cref{thm:nonpoly}, since $E_\KK$ lives on a hyperplane at
positive distance from the origin.

\begin{example}[second-order cone]\label{ex:soc}
Let $\mathcal{L}^n=\{(x,t)\in\RR^{n-1}\times\RR:\norm{x}\le t\}$.
The extreme rays of~$\mathcal{L}^n$ are
$\RR_+(\omega,\norm{\omega})$ with
$\omega\in\RR^{n-1}\setminus\{0\}$, so after normalization
\begin{align}
E_{\mathcal{L}^n}=\bigl\{(\omega/\sqrt 2,1/\sqrt 2):\omega\in S^{n-2}\bigr\}\subset\bigl\{(x,t):t=1/\sqrt 2\bigr\}.
\end{align}
Hence $\closu\aff(E_{\mathcal{L}^n})=\RR^{n-1}\times\{1/\sqrt 2\}$,
the projection of the origin onto this hyperplane is
$(0,1/\sqrt 2)$, and
\begin{align}
d=(0,-1/\sqrt 2),\qquad\norm{d}^2=\tfrac{1}{2}.
\end{align}
A direct check confirms~\cref{thm:nonpoly}: writing $v=(v_x,v_t)$,
the inclusion $d+v\in(\mathcal{L}^n)^\circ=-\mathcal{L}^n$ amounts
to $v_t+\norm{v_x}\le 1/\sqrt 2$, and the worst case
$v_t=\norm{v_x}=\norm{v}/\sqrt 2$ gives
$v_t+\norm{v_x}=\sqrt 2\norm{v}\le 1/\sqrt 2$ exactly when
$\norm{v}\le \tfrac12$. The case $n=3$ is depicted
in~\cref{fig:soc}.
\end{example}

\begin{figure}[ht]
    \centering
    \includegraphics[width=0.8\textwidth]{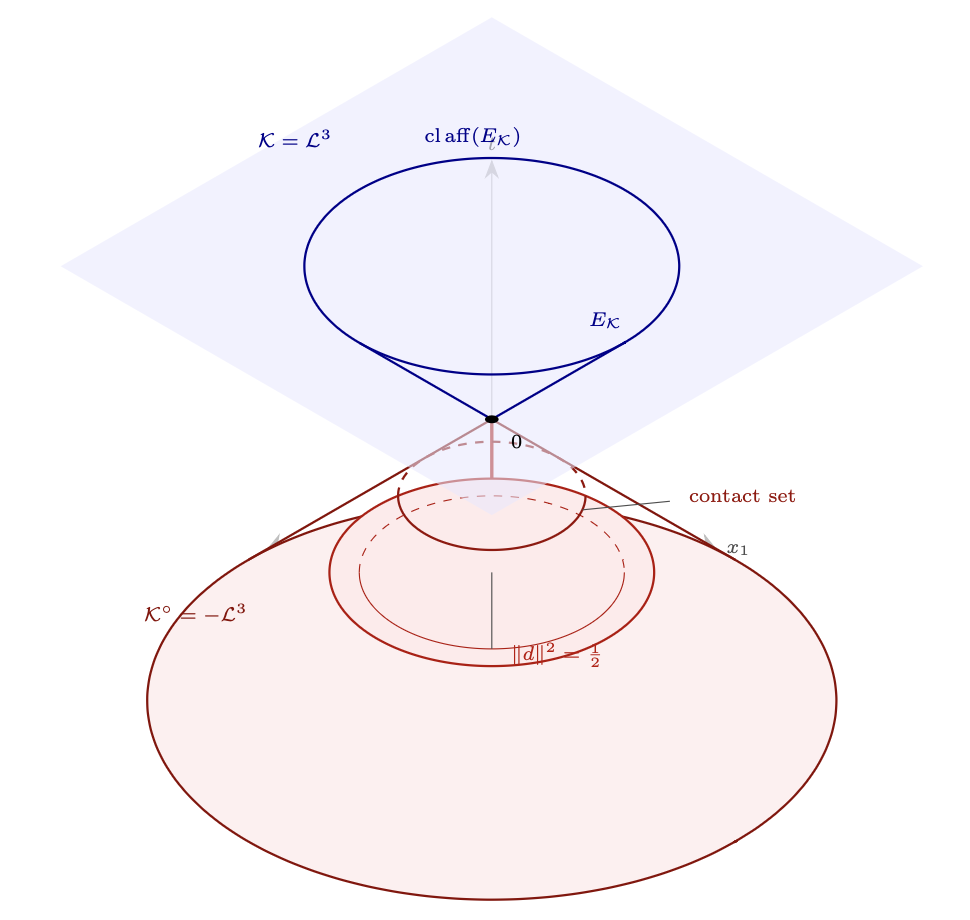}
    \caption{The second-order cone case of \cref{ex:soc} for $n=3$. The
upper (blue) cone is $\KK=\mathcal{L}^3$; its normalized
extreme-ray set $E_\KK$ is the unit circle on the cone boundary at
height $t=1/\sqrt 2$. The horizontal plane through that circle is
$\closu\aff(E_\KK)$; it lies at positive distance from the origin,
so the hypothesis of \cref{thm:nonpoly} is satisfied. The polar
cone $\Kpolar=-\mathcal{L}^3$ (red, opening downward) contains the
projection-based direction $d=(0,0,-1/\sqrt 2)$ and the inscribed
Euclidean ball of $3$D radius $\norm{d}^2=\tfrac12$ around it. The
ball is internally tangent to the lateral surface of $\Kpolar$
along the circle at height $t=-1/(2\sqrt 2)$, the contact set
predicted by \cref{thm:sharp}. Front halves of the equator and
contact circles are drawn solid; back halves dashed.}
\label{fig:soc}
\end{figure}
In the second-order case, $\norm{d}^2=\tfrac12$ is independent of
$n$, and the inscribed-ball bound is attained on a whole
$(n-2)$-dimensional sphere, the locus
$\{v:v_t=\norm{v_x}=1/(2\sqrt 2)\}$, which becomes the dashed
tangent circle of~\cref{fig:soc} when $n=3$. The
positive-semidefinite case behaves differently: the dimension
enters the value of $\norm{d}^2$ explicitly.

\begin{example}[positive-semidefinite cone]\label{ex:psd}
Equip the symmetric matrices $\Smat^n$ with the Frobenius inner
product $\scal{X}{Y}=\Tr(XY)$. Extreme rays of $\Smat^n_+$ are
spanned by rank-one positive semidefinite matrices, so
\begin{align}
E_{\Smat^n_+}=\bigl\{vv^\top/\norm{v}^2:v\in\RR^n\setminus\{0\}\bigr\}\subset\{X\in\Smat^n:\Tr X=1\}.
\end{align}
The set inclusion is in fact an equality of affine hulls:
$\{vv^\top:\norm{v}=1\}$ spans $\Smat^n$ as a real vector space,
since taking $v=e_i$ recovers $e_ie_i^\top$ and
$v=(e_i+e_j)/\sqrt 2$ for $i\ne j$ recovers
$\tfrac12(e_ie_i^\top+e_je_j^\top)+\tfrac12(e_ie_j^\top+e_je_i^\top)$,
so the symmetric off-diagonal generators lie in the linear span.
The differences therefore span $\{X\in\Smat^n:\Tr X=0\}$, and
adding any single $vv^\top$ recovers the trace-one hyperplane:
$\aff(E_{\Smat^n_+})=\{X\in\Smat^n:\Tr X=1\}$. The Frobenius-closest
matrix in this hyperplane to the origin is $I/n$, giving
\begin{align}
d=-I/n,\qquad\norm{d}^2=\tfrac{1}{n}.
\end{align}
The conclusion of~\cref{thm:nonpoly} is then exactly the statement
that $I/n-V\succeq 0$ whenever $\norm{V}_F\le 1/n$, which is
confirmed by Weyl's inequality~\cite{HornJohnson:2013}:
$\lambda_{\min}(I/n-V)\ge 1/n-\norm{V}_2\ge 1/n-\norm{V}_F\ge 0$.
\end{example}

The orthant of \cref{ex:orthant}, the second-order cone of
\cref{ex:soc}, and the positive-semidefinite cone of \cref{ex:psd}
are the three classical instances of {simple symmetric cones},
that is, cones of squares in simple Euclidean Jordan algebras.
(The orthant case extends to $\RR^n_+\subset\RR^n$ with
$E_\KK=\{e_1,\ldots,e_n\}$, $d=-(1/n)\sum_i e_i$, and
$\norm{d}^2=1/n$.) The next proposition treats all three uniformly
through the Pierce decomposition and gives the closed form
$\norm{d}^2=1/r$ in the Jordan rank $r$.

\begin{proposition}[symmetric cones satisfy the hypothesis]\label{prop:symmetric}
Let $V$ be a simple Euclidean Jordan algebra of rank $r$ with unit
element $e$ and Jordan trace $\tau$, equipped with the trace inner
product $\scal{x}{y}_{\!J}:=\tau(x\circ y)$, and let $\KK\subset V$
be the cone of squares. Then{\em :}
\begin{itemize}
\item[ {\bf (i)}] The unit-normalized extreme-ray set $E_\KK$ (with
respect to $\scal{\cdot}{\cdot}_{\!J}$) is the set of primitive
idempotents of~$V$.
\item[ {\bf (ii)}] $E_\KK\subset\{x\in V:\tau(x)=1\}$, hence
$0\notin\aff(E_\KK)$ and the hypothesis of \cref{thm:nonpoly} is
satisfied.
\item[ {\bf (iii)}] $d=-e/r$ and $\norm{d}_{\!J}^2=1/r$.
\end{itemize}
\end{proposition}

\begin{proof}
(i) Extreme rays of the cone $\KK$ are spanned by primitive
idempotents \cite[Prop.~III.1.4]{FK:1994}. A primitive idempotent
$c$ satisfies $c\circ c=c$, so $\norm{c}_{\!J}^2=\tau(c\circ c)=\tau(c)=1$
by $\tau(c)=1$ for primitive idempotents in a rank-$r$ Jordan
algebra \cite[Prop.~III.1.5]{FK:1994}. Hence primitive idempotents
are already $\scal{\cdot}{\cdot}_{\!J}$-unit and form $E_\KK$.

(ii) Immediate from~(i) and $\tau(c)=1$.

(iii) The Jordan spectral theorem gives a Jordan frame
$\{c_1,\ldots,c_r\}$ of orthogonal primitive idempotents with
$\sum_{i=1}^r c_i=e$ \cite[Thm.~III.1.2]{FK:1994}. Hence
$e/r=(1/r)\sum c_i\in\conv(E_\KK)\subset\aff(E_\KK)$. We claim
$-d=e/r$. Since $\aff(E_\KK)\subset\{x:\scal{x}{e}_{\!J}=1\}$ (using
$\tau(x)=\scal{x}{e}_{\!J}$ as $x\circ e=x$), it suffices to show
that $e/r$ is the orthogonal projection of $0$ onto the hyperplane
$H:=\{x:\scal{x}{e}_{\!J}=1\}$, since the projection of $0$ onto a
smaller affine set containing $\proj_H(0)$ is itself. Lagrange
optimality on $H$ gives the projection in the form $\lambda e$
with $\lambda\norm{e}_{\!J}^2=1$, and
$\norm{e}_{\!J}^2=\tau(e\circ e)=\tau(e)=r$, so $\lambda=1/r$ and
$\proj_H(0)=e/r$. Therefore $-d=e/r$, $d=-e/r$, and
$\norm{d}_{\!J}^2=(1/r^2)\norm{e}_{\!J}^2=r/r^2=1/r$.
\end{proof}

\Cref{prop:symmetric} unifies the three classical cases through the
Jordan rank $r$. With their canonical inner products:

For $\KK=\RR^n_+$, the Jordan algebra $V=\RR^n$ has
componentwise product, $\tau(x)=\sum_i x_i$, and
$\scal{\cdot}{\cdot}_{\!J}$ coincides with the standard Euclidean
inner product. Rank $r=n$ gives $d=-(1,\ldots,1)/n$ and
$\norm{d}^2=1/n$.

For $\KK=\Smat^n_+$, the Jordan algebra $V=\Smat^n$ has
product $X\circ Y=(XY+YX)/2$, $\tau(X)=\Tr X$, and
$\scal{\cdot}{\cdot}_{\!J}$ coincides with the Frobenius inner
product. Rank $r=n$ gives $d=-I/n$ and $\norm{d}^2=1/n$, recovering
\cref{ex:psd}.

For $\KK=\mathcal{L}^n$, the Lorentz Jordan algebra has
$(x,t)\circ(y,s)=(ty+sx,\scal{x}{y}+ts)$, unit element $e=(0,1)$,
$\tau((x,t))=2t$, and $\scal{\cdot}{\cdot}_{\!J}=2\scal{\cdot}{\cdot}$
on $\RR^n$. Rank $r=2$ gives $\norm{d}_{\!J}^2=1/2$. With Euclidean
normalization of $E_\KK$ (as in \cref{ex:soc}, which uses
$\scal{\cdot}{\cdot}$ rather than $\scal{\cdot}{\cdot}_{\!J}$), the
direction $d=(0,-1/\sqrt 2)$ has $\norm{d}^2=1/2=1/r$ as well,
consistent with the proposition up to the inner-product
normalization.

The proposition extends, with the same Jordan-frame proof, to the
remaining simple symmetric cones: the Hermitian complex and
quaternionic positive-semidefinite cones, and the exceptional
$27$-dimensional Albert cone. In every case
$\norm{d}_{\!J}^2=1/r$.

\subsection{Beyond symmetric cones: scope of the hypothesis}\label{sec:scope}

\Cref{prop:symmetric} settles the symmetric case but says nothing
about cones that are not self-dual or not homogeneous. We record
in this subsection three concrete results that test the reach of
the hypothesis $0\notin\closu\aff(E_\KK)$: a stability property
under direct products, a positive verification for the doubly
nonnegative cone (which is non-symmetric for $n\ge 3$), and an
explicit obstruction for the $p$-cones with $p\ne 2$.

\begin{proposition}[stability under direct products]\label{prop:dirsum}
Let $\KK_1\subset\RR^{n_1}$ and $\KK_2\subset\RR^{n_2}$ be
nontrivial closed convex pointed cones with
$0\notin\closu\aff(E_{\KK_1})$ and $0\notin\closu\aff(E_{\KK_2})$,
and let $d_1,d_2$ be the corresponding circumcentric directions
with squared norms $\delta_1:=\norm{d_1}^2$, $\delta_2:=\norm{d_2}^2$.
Then $\KK_1\times\KK_2\subset\RR^{n_1+n_2}$ satisfies
$0\notin\closu\aff(E_{\KK_1\times\KK_2})$, with circumcentric
direction
\begin{align}\label{eq:dirsum-d}
d=\Bigl(\frac{\delta_2}{\delta_1+\delta_2}\,d_1,\ \frac{\delta_1}{\delta_1+\delta_2}\,d_2\Bigr)
\end{align}
and squared norm
\begin{align}\label{eq:dirsum-norm}
\frac{1}{\norm{d}^2}=\frac{1}{\delta_1}+\frac{1}{\delta_2}.
\end{align}
\end{proposition}

\begin{proof}
The extreme rays of $\KK_1\times\KK_2$ are $\RR_+(u,0)$ for
$u\in E_{\KK_1}$ together with $\RR_+(0,v)$ for $v\in E_{\KK_2}$,
both already Euclidean unit, so
$E_{\KK_1\times\KK_2}=(E_{\KK_1}\times\{0\})\cup(\{0\}\times E_{\KK_2})$.
Closing under affine combinations,
\begin{align*}
\closu\aff(E_{\KK_1\times\KK_2})=\bigl\{(\alpha\, u,\,(1-\alpha)\, v):
u\in\closu\aff(E_{\KK_1}),\ v\in\closu\aff(E_{\KK_2}),\ \alpha\in\RR\bigr\}.
\end{align*}
The squared norm of such a point is
$\alpha^2\norm{u}^2+(1-\alpha)^2\norm{v}^2$, which is minimized
over $u\in\closu\aff(E_{\KK_1})$ at $u=-d_1$ with
$\norm{u}^2=\delta_1$, and similarly for $v$ at $-d_2$ with
$\norm{v}^2=\delta_2$. The remaining one-dimensional minimization
over $\alpha$ of $\alpha^2\delta_1+(1-\alpha)^2\delta_2$ gives
$\alpha=\delta_2/(\delta_1+\delta_2)$ and minimum value
$\delta_1\delta_2/(\delta_1+\delta_2)$. Hence
$-d=\bigl(\alpha\,(-d_1),(1-\alpha)\,(-d_2)\bigr)$ and
$\norm{d}^2=\delta_1\delta_2/(\delta_1+\delta_2)$, which
is~\cref{eq:dirsum-d}--\cref{eq:dirsum-norm}. The minimizer is
nonzero, so $0\notin\closu\aff(E_{\KK_1\times\KK_2})$.
\end{proof}

\Cref{prop:dirsum} iterates immediately: a finite product
$\KK_1\times\cdots\times\KK_L$ of cones satisfying the hypothesis
satisfies it, with $1/\norm{d}^2=\sum_{\ell=1}^L 1/\delta_\ell$,
the parallel-resistance formula. The orthant
$\RR^n_+=(\RR_+)^n$ falls out as the $L=n$ case with
$\delta_\ell=1$, recovering $\norm{d}^2=1/n$ from
\cref{prop:symmetric}.

\begin{example}[doubly nonnegative cone]\label{ex:dnn}
The {\em doubly nonnegative cone}
$$\mathrm{DNN}^n=\Smat^n_+\cap(\RR_{\ge 0})^{n\times n}$$ consists
of positive-semidefinite matrices with all entries nonnegative.
It is non-polyhedral and, for $n\ge 3$, non-symmetric: it has no
transitive automorphism group on its interior, and it does not
coincide with the cone of squares in any Euclidean Jordan
algebra~\cite{FK:1994}. The extreme rays of $\mathrm{DNN}^n$ are
spanned by rank-one matrices $vv^\top$ with $v\in\RR^n_+$ (see,
e.g., the description of completely positive and doubly nonnegative
cones in Berman and Shaked-Monderer~\cite{BSM:2003}). Equipped with
the Frobenius inner product, the unit-normalized extremal section
is
\begin{align}
E_{\mathrm{DNN}^n}=\bigl\{vv^\top:v\in\RR^n_+,\ \norm{v}=1\bigr\},
\end{align}
which lies on the trace-one hyperplane
$\{X\in\Smat^n:\Tr X=1\}$, since $\Tr(vv^\top)=\norm{v}^2=1$. The
closed affine hull of $E_{\mathrm{DNN}^n}$ is contained in this
hyperplane, hence avoids the origin, and the hypothesis of
\cref{thm:nonpoly} is satisfied. The matrix
$I/n=(1/n)\sum_i e_ie_i^\top$ is a convex combination of
$e_ie_i^\top\in E_{\mathrm{DNN}^n}$, so it lies in
$\aff(E_{\mathrm{DNN}^n})$, and the Frobenius-closest matrix in
the trace-one hyperplane to the origin is $I/n$ (Lagrange
optimality, as in~\cref{ex:psd}). Hence
\begin{align}
d=-I/n,\qquad\norm{d}^2=\tfrac{1}{n},
\end{align}
the same value as for $\Smat^n_+$. \Cref{thm:nonpoly} states that
every symmetric $V\in\Smat^n$ with $\norm{V}_F\le 1/n$ satisfies
$V-I/n\in(\mathrm{DNN}^n)^\circ$, that is,
$\Tr\bigl((V-I/n)X\bigr)\le 0$ for every $X\in\mathrm{DNN}^n$.
The polar of an intersection contains the union of the polars,
so $(\mathrm{DNN}^n)^\circ\supseteq -\Smat^n_+\cup-(\RR_{\ge 0})^{n\times n}$;
the unified construction therefore captures certificates that
neither of the two defining cones provides on its own.
\end{example}

The example illustrates a useful invariance: $\norm{d}^2$ depends
only on the affine geometry of $E_\KK$, not on the algebraic
structure of $\KK$. The doubly nonnegative cone has the same
$\norm{d}^2$ as the positive-semidefinite cone in which it sits,
even though their interiors differ markedly.

On the obstruction side, the polyhedral counterexample
\cref{ex:degen} below shows that the hypothesis can fail outright.
The next proposition shows that the failure persists even for
smooth, homogeneous, and full-dimensional cones, as soon as one
steps outside the symmetric setting.

\begin{proposition}[$p$-cones do not satisfy the hypothesis for $p\ne 2$]\label{prop:pcone}
For $1<p<\infty$ and $n\ge 3$, the $p$-cone
\begin{align}
\mathcal{L}^n_p:=\bigl\{(x,t)\in\RR^{n-1}\times\RR:\norm{x}_p\le t\bigr\}
\end{align}
satisfies $0\in\aff(E_{\mathcal{L}^n_p})$ if and only if $p\ne 2$.
\end{proposition}

\begin{proof}
The extreme rays of $\mathcal{L}^n_p$ are
$\RR_+(\omega,\norm{\omega}_p)$ for
$\omega\in\RR^{n-1}\setminus\{0\}$. Restricting to Euclidean-unit
$\omega\in S^{n-2}$, the unit-normalized extremal section is
\begin{align}
E_{\mathcal{L}^n_p}=\Bigl\{\bigl(\omega,\,\norm{\omega}_p\bigr)\big/\sqrt{1+\norm{\omega}_p^2}:\omega\in S^{n-2}\Bigr\}.
\end{align}
For $p=2$, the value $\norm{\omega}_2=1$ is constant in $\omega$,
so the last coordinate of every $u\in E_{\mathcal{L}^n_p}$ equals
$1/\sqrt 2$ and $\closu\aff(E_{\mathcal{L}^n_p})$ is the
hyperplane $\{t=1/\sqrt 2\}$, recovering~\cref{ex:soc}.

For $p\ne 2$, the function $\omega\mapsto\norm{\omega}_p$ is not
constant on $S^{n-2}$, since $\norm{e_1}_p=1$ while
$\norm{(e_1+e_2)/\sqrt 2}_p=2^{1/p-1/2}\ne 1$. The map
$\omega\mapsto(\omega,\norm{\omega}_p)/\sqrt{1+\norm{\omega}_p^2}$
from $S^{n-2}$ into $\RR^n$ is therefore not contained in any
hyperplane: any such hyperplane $H=\{x:\scal{x}{\nu}=c\}$ with
$\nu=(a,b)$ would force
$\scal{\omega}{a}+b\norm{\omega}_p=c\sqrt{1+\norm{\omega}_p^2}$
on $S^{n-2}$, and evaluating at
$\omega=\pm e_1, \pm e_2,(e_1+e_2)/\sqrt 2$ yields a system whose
only solution is $a=b=c=0$. Hence
$\closu\aff(E_{\mathcal{L}^n_p})=\RR^n\ni 0$.
\end{proof}

\begin{remark}[the gap is genuine]\label{rem:pcone-gap}
The obstruction is not that $\mathcal{L}^n_p$ lacks a useful
circumcentric direction. Projecting the origin onto the affine
hull of any compact base of $\mathcal{L}^n_p$ (say the slice
$\{t=1\}$) yields one. The point is that this direction is not
the projection onto $\closu\aff(E_{\mathcal{L}^n_p})$, which is
all of $\RR^n$ when $p\ne 2$. The hypothesis is therefore a
property of the chosen extremal section rather than of the cone
itself, and the construction of \cref{thm:nonpoly} is genuinely
sharper than what generic compact bases support. Identifying the
broadest class of non-symmetric cones for which a self-dual
projection target survives is an open problem, related to the
spectral cones and the homogeneous cones with rational
characteristic exponent~\cite{FK:1994,DW:2017}.
\end{remark}

The polyhedral counterexample below complements the smooth
obstruction of \cref{prop:pcone}: it shows that even in dimension
three and with a finite extreme-ray set, the hypothesis can fail.

\begin{example}[a pointed cone where the hypothesis fails]\label{ex:degen}
Take $$\KK=\cone\{e_1,e_2,e_3,u^4\}\subset\RR^3$$ with
$u^4:=(1,1,-1)/\sqrt 3$. The vector $\mu=(2,2,1)$ satisfies
$\scal{x}{\mu}>0$ on $\KK\setminus\{0\}$, so $\KK$ is pointed. A
direct check shows that none of the four generators lies in the
conic hull of the other three, so all four are extreme rays and
$E_\KK=\{e_1,e_2,e_3,u^4\}$. The differences $e_2-e_1$, $e_3-e_1$,
$u^4-e_1$ form a $3\times 3$ matrix of determinant
$1/\sqrt 3-1\ne 0$, so the four points are affinely independent
in $\RR^3$ and $\aff(E_\KK)=\RR^3\ni 0$ (see~\cref{fig:degen}).
The construction~\cref{eq:dnonpoly} therefore returns $d=0$,
\cref{thm:nonpoly} reduces to the trivial $v=0$, and the
inscribed-ball estimate is vacuous. By \cref{lem:hyp}, no nonzero
$\nu$ can have constant inner product with all four generators of
$E_\KK$, although $\mu$ above gives a uniformly positive lower
bound. Pointedness is therefore strictly weaker than the
hypothesis of~\cref{thm:nonpoly}.
\end{example}

\begin{figure}[ht]
    \centering
    \includegraphics[width=0.5\textwidth]{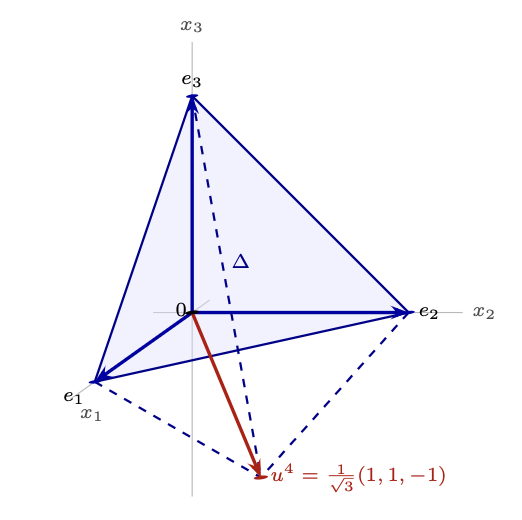}
    \caption{The non-example of \cref{ex:degen}. The first three
normalized extreme rays $e_1,e_2,e_3$ lie on the simplex plane
$\Delta=\{x_1+x_2+x_3=1\}$ (shaded), but $u^4$ has component sum
$1/\sqrt 3$ and therefore lies on the strictly parallel plane
$\{x_1+x_2+x_3=1/\sqrt 3\}$, off~$\Delta$. The four points are
then affinely independent in~$\RR^3$, their convex hull is a
non-degenerate tetrahedron (edges to $u^4$ dashed), and
$\aff(E_\KK)=\RR^3\ni 0$. The projection of the origin onto
$\closu\aff(E_\KK)$ is the origin itself, so~\cref{eq:dnonpoly}
returns $d=0$ and the inscribed-ball conclusion of~\cref{thm:nonpoly}
is vacuous, even though $\KK$ is pointed.}
\label{fig:degen}
\end{figure}   
\begin{remark}[Hilbert-space version]\label{rem:hilbert}
The proof of~\cref{thm:nonpoly} uses only orthogonal projection
onto a closed affine subspace, the Cauchy--Schwarz inequality, and
the representation of~$\KK$ as the closed conic hull of~$E_\KK$.
The first two hold in any Hilbert space~$\Hcal$, and the third
follows from the Choquet integral
representation~\cite{Phelps:2001} whenever~$\KK\subset\Hcal$ has a
weakly compact base. The same argument therefore yields
$\scal{d+v}{z}\le 0$ for every $z\in \KK$ in that setting.
\end{remark}

\section{The maximum interior step}\label{sec:step}

The closing remark of~\cite{BehlingEtAl_CDC} suggests using
circumcentric directions inside algorithms for constrained convex
optimization, but stops short of specifying step lengths.
\Cref{cor:dirdepth} fills that gap. Specialized to the active-cone
setting of~\cite[Cor.~2.7--2.8]{BehlingEtAl_CDC}, it yields a
closed-form expression for the largest step that keeps an iterate
strictly inside $\Omega$ when one moves from~$d$ along an arbitrary
direction.

\begin{proposition}[sharp interior step]\label{prop:step}
Let $\Omega=\{x\in\RR^n:g(x)\le 0\}$, where
$g\colon\RR^n\to\RR^m$ has convex differentiable components and
satisfies the Slater condition. Fix $\bar x\in\Omega$ with
$J(\bar x):=\{j:g_j(\bar x)=0\}\ne\emptyset$, take
$B_\KK=\{u^j\}_{j\in J(\bar x)}$ to be the normalized conic base
of $\KK=\cone\{\nabla g_j(\bar x):j\in J(\bar x)\}$, and set
$d=-\circum(B_\KK)$. Then for every
$w\in\RR^n\setminus\{0\}$ the quantity
\begin{align}\label{eq:sigmastar}
\sigma^{\star}(w):=
\begin{cases}
+\infty, & \max_{j\in J(\bar x)}\scal{w}{u^j}\le 0,\\[1mm]
\dfrac{\norm{d}^2}{\max_{j\in J(\bar x)}\scal{w}{u^j}}, & \text{otherwise,}
\end{cases}
\end{align}
is the supremum of $\sigma\ge 0$ for which $d+\sigma w$ remains an
interior feasible direction for~$\Omega$ at~$\bar x$.
\end{proposition}

\begin{proof}
Apply~\cref{cor:dirdepth} to~$w$: $d+\sigma w\in\inte\Kpolar$ for
every $\sigma\in[0,\sigma^\star(w))$, with the binding constraint
at $\sigma=\sigma^\star(w)$. Slater forces
$\Kpolar=\Tcone_\Omega(\bar x)$ \cite[Cor.~2.8]{BehlingEtAl_CDC},
so $\inte\Kpolar$ is the cone of interior feasible directions
at~$\bar x$.
\end{proof}

A natural application is a feasibility-corrected projected
gradient update. Given $\bar x\in\Omega$ and a descent
direction~$w$, for instance $w=-\nabla f(\bar x)$ for an
unconstrained objective~$f$, the step
\begin{align}\label{eq:fcpg}
\bar x_+:=\bar x+t\bigl(d+\sigma w\bigr),\qquad
\sigma\in\bigl(0,\sigma^\star(w)\bigr),\ t>0\text{ small,}
\end{align}
is feasible by construction. Two choices of $\sigma$ are common.
The choice $\sigma=\tfrac12\sigma^\star(w)$ preserves a uniform
interior margin and is suited to the case where
$\sigma^\star(w)$ is already computed. The conservative
$\sigma=\norm{d}^2/\norm{w}$ avoids the inner maximization over
$J(\bar x)$ and remains valid by Cauchy--Schwarz; it is
preferable when $w$ is close to the axis $-d/\norm{d}$, an angle
controlled by~\cref{cor:angular}.

\begin{remark}[scale invariance]\label{rem:nonormal}
The formula~\cref{eq:sigmastar} is positively homogeneous of
degree~$-1$ in~$w$: $\sigma^\star(\alpha w)=\sigma^\star(w)/\alpha$
for $\alpha>0$. The displacement $\sigma^\star(w)w$ depends
therefore only on the direction of~$w$, not on its norm.
\end{remark}

\section{Piecewise smooth constraints}\label{sec:piecewise}

The smoothness assumption on $g$
in~\cite[Cor.~2.8]{BehlingEtAl_CDC} can be relaxed at no cost when
each $g_j$ is the maximum of finitely many smooth convex
functions. This setting includes cutting-plane formulations,
piecewise quadratic penalties, and standard descriptions of
$L_\infty$-balls $\norm{Ax-b}_\infty\le\tau$. The same analysis
applies to the cone generated by the gradients of the active
pieces.

\begin{corollary}[piecewise smooth admissible step]\label{cor:piecewise}
Suppose $$g_j(x)=\max_{i\in I_j}g_{j,i}(x)$$ with each~$g_{j,i}$
convex and continuously differentiable on~$\RR^n$ and each~$I_j$
finite. Fix $\bar x\in\Omega=\{g\le 0\}$, set
$J(\bar x)=\{j:g_j(\bar x)=0\}$ and
$I_j(\bar x)=\{i\in I_j:g_{j,i}(\bar x)=0\}$ for $j\in J(\bar x)$,
and define the active-pieces cone
\begin{align}
\KK:=\cone\bigl\{\nabla g_{j,i}(\bar x):j\in J(\bar x),\ i\in I_j(\bar x)\bigr\}.
\end{align}
Let $B_\KK$ denote the normalized conic base of~$\KK$ and
$d=-\circum(B_\KK)$. Suppose the inner Slater condition
\begin{align}
\bigl\{x\in\RR^n:\scal{\nabla g_{j,i}(\bar x)}{x-\bar x}<0\text{ for all }j\in J(\bar x),\ i\in I_j(\bar x)\bigr\}\ne\emptyset
\end{align}
holds. Then $d+v$ is an interior feasible direction for~$\Omega$
at~$\bar x$ for every $\norm{v}<\norm{d}^2$.
\end{corollary}

\begin{proof}
The cone~$\KK$ is finitely generated, so the analysis of
\Cref{sec:exact,sec:step} applies and gives
$d+v\in\inte\Kpolar$ for $\norm{v}<\norm{d}^2$. The inner Slater
condition together with convexity of each $g_{j,i}$ yields
$\Kpolar=\Tcone_\Omega(\bar x)$, and the conclusion follows from
$\inte\Kpolar\subset\Tcone_\Omega(\bar x)$.
\end{proof}

\begin{remark}[the inner Slater condition is MFCQ]\label{rem:mfcq}
Under convexity of the pieces, the Clarke
subdifferential~\cite{Clarke:1983} of $g_j$ at
$\bar x\in\{g_j=0\}$ is
$\partial g_j(\bar x)=\conv\{\nabla g_{j,i}(\bar x):i\in I_j(\bar x)\}$,
so the inner Slater condition stated in \cref{cor:piecewise} is
the existence of $d\in\RR^n$ with $\scal{w}{d}<0$ for every
$w\in\partial g_j(\bar x)$ and every $j\in J(\bar x)$. This is
exactly the Mangasarian--Fromovitz constraint
qualification~\cite{Mangasarian:1967} at $\bar x$ for the
(nonsmooth) system $\{g_j\le 0\}_{j\in J(\bar x)}$.
\end{remark}

For a polyhedral illustration, take
$\Omega=\{x:\norm{Ax-b}_{\infty}\le\tau\}$ for a matrix
$A\in\RR^{p\times n}$, written componentwise as
$g_j(x)=\max\bigl\{e_j^\top(Ax-b)-\tau,\ -e_j^\top(Ax-b)-\tau\bigr\}$.
The gradients of the two active pieces at any tight constraint
are $\pm A^\top e_j$, and the corresponding active-pieces cone is
generated by the signed rows of $A$. The next section develops
this case in full and treats a parallel SOC instance.

\section{Two concrete problem classes}\label{sec:concrete}

The constructions of \Cref{sec:step,sec:piecewise} prescribe a
feasible direction and a step length given an active-pieces cone,
but stop short of computation. We show in this section that two
standard problem classes, $L_\infty$-ball constrained least
squares (piecewise smooth, polyhedral feasible set) and
second-order cone programming (smooth, non-polyhedral feasible
set), yield explicit closed forms for $d$, $\norm{d}^2$, and
$\sigma^\star$ through the Gram-matrix formula of
\cref{prop:gram}, with no auxiliary optimization. The two cases
together cover both branches of the framework:
\cref{cor:piecewise} for the first, \cref{prop:step} for the
second.

\subsection{$L_\infty$-ball constrained least squares}\label{sec:linf}

Consider the box-constrained problem
\begin{align}\label{eq:linfprob}
\min_{x\in\RR^n}\ \tfrac12\norm{Ax-b}^2\quad\text{subject to}\quad\norm{Cx-d}_\infty\le\tau,
\end{align}
with $A\in\RR^{m\times n}$, $b\in\RR^m$, $C\in\RR^{p\times n}$,
$d\in\RR^p$, and $\tau>0$. Writing $c_j$ for the $j$-th row of
$C$, the feasible set
\begin{align}
\Omega=\bigl\{x\in\RR^n:c_j^\top x-d_j\in[-\tau,\tau]\text{ for }j=1,\ldots,p\bigr\}
\end{align}
is described componentwise by the piecewise affine constraints
\begin{align}\label{eq:linfgj}
g_j(x)=\max\bigl\{c_j^\top x-d_j-\tau,\ -(c_j^\top x-d_j)-\tau\bigr\},
\end{align}
so~\cref{cor:piecewise} applies. At a feasible $\bar x$, define
the {\em signed} active set
\begin{align}\label{eq:linfJpm}
J^{\pm}(\bar x):=\bigl\{(j,\epsilon_j):c_j^\top\bar x-d_j=\epsilon_j\tau,\ \epsilon_j\in\{\pm 1\}\bigr\}.
\end{align}
Each index $j$ contributes at most one signed pair: the
constraints at $+\tau$ and $-\tau$ are simultaneously active only
in the degenerate case $\tau=0$. The active-pieces cone is the
polyhedral cone
\begin{align}\label{eq:linfK}
\KK=\cone\bigl\{\epsilon_j c_j:(j,\epsilon_j)\in J^{\pm}(\bar x)\bigr\}\subset\RR^n,
\end{align}
with normalized base
$u^{(j,\epsilon_j)}:=\epsilon_j c_j/\norm{c_j}$, and the Gram
matrix
\begin{align}\label{eq:linfgram}
M_{(j,\epsilon_j),(k,\epsilon_k)}=\frac{\epsilon_j\epsilon_k\scal{c_j}{c_k}}{\norm{c_j}\norm{c_k}}
\end{align}
delivers, via~\cref{prop:gram}, the closed forms
\begin{align}\label{eq:linfd}
d=-\frac{1}{\one^\top M^{-1}\one}\sum_{(j,\epsilon_j)\in J^{\pm}(\bar x)}(M^{-1}\one)_{(j,\epsilon_j)}\,\frac{\epsilon_j c_j}{\norm{c_j}},\qquad
\norm{d}^2=\frac{1}{\one^\top M^{-1}\one}.
\end{align}

The unconstrained descent direction is
$w=-\nabla f(\bar x)=-A^\top(A\bar x-b)$, and~\cref{cor:dirdepth}
returns the maximum interior step
\begin{align}\label{eq:linfsigma}
\sigma^{\star}(w)=
\begin{cases}
+\infty, & \displaystyle\max_{(j,\epsilon_j)\in J^\pm(\bar x)}\epsilon_j\scal{c_j}{w}/\norm{c_j}\le 0,\\[2mm]
\dfrac{\norm{d}^2}{\displaystyle\max_{(j,\epsilon_j)\in J^\pm(\bar x)}\epsilon_j\scal{c_j}{w}/\norm{c_j}}, & \text{otherwise.}
\end{cases}
\end{align}
The iterate $\bar x_+:=\bar x+t(d+\sigma w)$ remains in the
interior of $\Omega$ for every
$\sigma\in(0,\sigma^{\star}(w))$ and every sufficiently small
$t>0$, by~\cref{cor:piecewise}.

Two regimes bound the geometry. When the active rows
$\{c_j\}_{(j,\cdot)\in J^\pm(\bar x)}$ are mutually
orthonormal, the case $C=I_p$ in particular, where each row is
a coordinate direction, the Gram matrix $M$ is the identity and
\cref{eq:linfd} collapses to
\begin{align}\label{eq:linforthant}
d=-\frac{1}{|J^{\pm}|}\sum_{(j,\epsilon_j)\in J^{\pm}(\bar x)}\epsilon_j c_j,\qquad\norm{d}^2=\frac{1}{|J^{\pm}|},
\end{align}
recovering the orthant geometry of \cref{ex:orthant} up to signs
and giving the $1/|J^{\pm}|$ scaling for vertex-active iterates of
the $L_\infty$ ball. When the active rows approach linear
dependence, the spectral bound
$\norm{d}^2\ge\lambda_{\min}(M)/|J^{\pm}|$ of~\cref{prop:gram}
forces $\norm{d}^2\to 0$ at the rate of $\lambda_{\min}(M)$, and
the inscribed-ball estimate degrades accordingly. The aperture
identity $\norm{d}=\cos\theta(B_\KK)$ of \cref{prop:aperture}
reads, in the orthonormal regime,
$\theta(B_\KK)=\arccos(1/\sqrt{|J^{\pm}|})$: the half-aperture
grows as more constraints saturate, and the directional-depth
bound~\cref{eq:rhobound} of \cref{cor:angular} yields a uniform
estimate depending only on $|J^{\pm}|$ and on the angle of the
descent direction $w$ to the axis $-d/\norm{d}$.

\subsection{Second-order cone programming}\label{sec:socp}

The companion smooth instance is the second-order cone program
\begin{align}\label{eq:socp}
\min_{x\in\RR^n}\ \tfrac12 x^\top Q x+q^\top x\quad\text{subject to}\quad\norm{A_jx-b_j}\le c_j^\top x+\delta_j,\ j=1,\ldots,m,
\end{align}
with $Q\succeq 0$, $A_j\in\RR^{n_j\times n}$, $b_j\in\RR^{n_j}$,
$c_j\in\RR^n$, $\delta_j\in\RR$, satisfying Slater's condition.
The non-polyhedral feasible set
\begin{align}
\Omega=\bigcap_{j=1}^{m}\bigl\{x\in\RR^n:\norm{A_jx-b_j}\le c_j^\top x+\delta_j\bigr\}
\end{align}
is the intersection of $m$ rotated second-order cones. Each
constraint $$g_j(x)=\norm{A_jx-b_j}-c_j^\top x-\delta_j$$ is convex
and continuously differentiable on $\{x:A_jx\ne b_j\}$, with
gradient
\begin{align}\label{eq:socgrad}
\nabla g_j(\bar x)=\frac{A_j^\top(A_j\bar x-b_j)}{\norm{A_j\bar x-b_j}}-c_j.
\end{align}
\Cref{prop:step} therefore applies at any feasible $\bar x$ with
$A_j\bar x\ne b_j$ for every $j\in J(\bar x)$, the {\em apex-free}
condition that excludes the lone non-smooth point of each SOC.

The geometric content is that, although $\Omega$ itself is
non-polyhedral, the active cone
$\KK=\cone\{\nabla g_j(\bar x):j\in J(\bar x)\}$ is finitely
generated (by $|J(\bar x)|$ vectors) at any apex-free feasible
iterate, and the polyhedral closed forms of
\Cref{sec:exact,sec:gram,sec:aperture} apply directly. Writing
$g^j:=\nabla g_j(\bar x)$ from~\cref{eq:socgrad} and
$u^j:=g^j/\norm{g^j}$, the Gram matrix
$M_{ij}=\scal{u^i}{u^j}$ has size
$|J(\bar x)|\times|J(\bar x)|$ and yields, via~\cref{prop:gram},
\begin{align}\label{eq:socd}
d=-\frac{1}{\one^\top M^{-1}\one}\sum_{j\in J(\bar x)}(M^{-1}\one)_j u^j,\qquad
\norm{d}^2=\frac{1}{\one^\top M^{-1}\one}.
\end{align}
For a descent direction $w=-(Q\bar x+q)$ of the quadratic
objective, \cref{prop:step} prescribes
\begin{align}\label{eq:socsigma}
\sigma^{\star}(w)=\frac{\norm{d}^2}{\max_{j\in J(\bar x)}\scal{w}{u^j}}
\end{align}
when the maximum is positive, and $\sigma^{\star}(w)=+\infty$
otherwise. The closed form~\cref{eq:socsigma} is the SOCP analogue
of~\cref{eq:linfsigma}, with the signed coordinate rows
$\epsilon_j c_j/\norm{c_j}$ of the $L_\infty$ case replaced by the
SOC active gradients~\cref{eq:socgrad}.

The single-constraint case sets the dimension-free benchmark.
With $m=1$, $J(\bar x)=\{1\}$, the cone $\KK$ is a single ray,
$d=-u^1$, $\norm{d}^2=1$, and~\cref{eq:socsigma} reduces to
$\sigma^{\star}(w)=1/\scal{w}{u^1}$ when this is positive. With
two active SOC constraints whose normalized gradients have inner
product $\scal{u^1}{u^2}=\rho\in(-1,1)$, the Gram matrix
$M=\bigl(\begin{smallmatrix}1&\rho\\\rho&1\end{smallmatrix}\bigr)$
gives $\one^\top M^{-1}\one=2/(1+\rho)$ and
\begin{align}\label{eq:soctwo}
\norm{d}^2=\frac{1+\rho}{2},
\end{align}
which interpolates between the orthogonal case $\rho=0$
($\norm{d}^2=1/2$, recovering~\cref{ex:orthant}), the co-aligned
limit $\rho\uparrow 1$ ($\norm{d}^2\uparrow 1$, collapsing to a
single ray), and the antipodal limit $\rho\downarrow-1$
($\norm{d}^2\downarrow 0$, the loss-of-pointedness regime
of~\cref{rem:cond}). The step length~\cref{eq:socsigma} inherits
the same continuous dependence on $\rho$ through $\norm{d}^2$,
giving the explicit two-constraint formula
\begin{align}\label{eq:soctwosigma}
\sigma^{\star}(w)=\frac{(1+\rho)/2}{\max\{\scal{w}{u^1},\scal{w}{u^2}\}},
\end{align}
which remains bounded as long as the two active gradients are not
exactly antipodal.

The role of \Cref{sec:nonpoly} for SOCP is qualitative: the
active-cone view collapses each SOC constraint to a single
normalized gradient, but the {\em underlying} cone
$\mathcal{L}^{n_j}$ contributes the global aperture estimates of
\cref{ex:soc} and the inscribed-ball margin $1/2$ on its polar,
which the active-cone $\norm{d}^2$ cannot violate uniformly along
an iterate sequence. The $1/2$ benchmark of \cref{ex:soc} thus
serves as a geometric reference for the SOCP step lengths
returned by~\cref{eq:socsigma}, and the Jordan-rank value
$1/r=1/2$ from \cref{prop:symmetric} confirms that this benchmark
is intrinsic to the second-order cone, not an artifact of the
active-cone specialization.

\section{A Bregman extension}\label{sec:bregman}

Identity~\cref{eq:keyident} is the optimality condition for an
orthogonal projection. Replacing the Euclidean projection with a
Bregman projection induced by a Legendre function~$h$ produces a
direction $d_h$ for which the same identity holds in dual
coordinates, and a Bregman inscribed-ball estimate follows by the
Cauchy--Schwarz argument of~\cref{thm:nonpoly}. The Euclidean case
$h(x)=\tfrac12\norm{x}^2$ is recovered. We follow the standard Bregman conventions;
see~\cite{BauschkeBorwein:1996,Combettes:1997,BC:book}. Throughout
this section, $h\colon\RR^n\to\RR\cup\{+\infty\}$ is a {Legendre}
function: proper, lower semicontinuous, convex, essentially smooth
and essentially strictly convex, with $\inte\dom h\ne\emptyset$.
The {Bregman divergence} is
\begin{align}\label{eq:bregdiv}
D_h(x,y):=h(x)-h(y)-\scal{\nabla h(y)}{x-y},\qquad x\in\dom h,\ y\in\inte\dom h,
\end{align}
and is non-negative, jointly continuous, and zero only on the
diagonal. The {Bregman projection} of $y\in\inte\dom h$ onto a
closed convex set $C\subset\dom h$ that meets $\inte\dom h$ is
the unique minimizer
\begin{align}
P^h_C(y):=\argmin_{x\in C\cap\inte\dom h}D_h(x,y),
\end{align}
when the minimum exists and is attained in the interior. The
Euclidean case $h(x)=\tfrac12\norm{x}^2$ has
$D_h(x,y)=\tfrac12\norm{x-y}^2$ and $P^h_C=\proj_C$.

We add the standing hypothesis that $0$ is the minimizer of $h$:
\begin{align}\label{eq:0min}
0\in\inte\dom h,\qquad \nabla h(0)=0.
\end{align}
This normalization makes $0$ the Bregman ``origin''. It holds for
$h(x)=\tfrac12\norm{x}^2$ and, more generally, for
$h(x)=\tfrac1p\norm{x}^p$ with $p\ge 2$ (with the convention
$\nabla h(0)=0$).

\begin{lemma}[Bregman key identity]\label{lem:bregman-key}
Let $\KK=\cone(B_\KK)\subset\RR^n$ be a polyhedral cone with
normalized conic base $B_\KK=\{u^1,\ldots,u^p\}$, and assume
$\aff(B_\KK)\subset\inte\dom h$ and $0\notin\aff(B_\KK)$. Set
\begin{align}\label{eq:ch}
c_h:=\argmin_{x\in\aff(B_\KK)}D_h(x,0)
\end{align}
(the unique minimizer in $\inte\dom h$ when it exists), and
$\kappa_h:=\scal{\nabla h(c_h)}{c_h}$. Then $c_h\ne 0$, $\kappa_h>0$,
and
\begin{align}\label{eq:breg-key}
\scal{\nabla h(c_h)}{u^i}=\kappa_h,\qquad i=1,\ldots,p.
\end{align}
\end{lemma}

\begin{proof}
Since $\nabla h(0)=0$, we have
$\nabla_x D_h(x,0)=\nabla h(x)-\nabla h(0)=\nabla h(x)$. The
first-order optimality condition for minimizing $D_h(\cdot,0)$
over the affine set $\aff(B_\KK)$ at the interior minimizer $c_h$
is $\nabla h(c_h)\perp\aff(B_\KK)-c_h$, equivalently
$\scal{\nabla h(c_h)}{x-c_h}=0$ for every $x\in\aff(B_\KK)$.
Setting $x=u^i$ gives~\cref{eq:breg-key} with
$\kappa_h=\scal{\nabla h(c_h)}{c_h}$. Strict monotonicity of
$\nabla h$ on $\inte\dom h$ (a consequence of essential strict
convexity~\cite{Rockafellar:1970,BC:book}) and $\nabla h(0)=0$
give $\scal{\nabla h(c_h)-\nabla h(0)}{c_h-0}>0$ whenever
$c_h\ne 0$, that is, $\kappa_h>0$. Finally $c_h\ne 0$ follows from
$0\notin\aff(B_\KK)$.
\end{proof}

Identity~\cref{eq:breg-key} is the Bregman counterpart
of~\cref{eq:keyident}: the dual point $\nabla h(c_h)$ has the same
inner product $\kappa_h$ with every generator of~$B_\KK$.
Cauchy--Schwarz in the primal yields the inscribed-ball estimate.

\begin{theorem}[Bregman inscribed-ball estimate]\label{thm:bregman-ball}
Under the hypotheses of \cref{lem:bregman-key}, set
$d_h:=-\nabla h(c_h)$. Then $\scal{d_h}{u^i}=-\kappa_h$ for each
$i$, and for every $v\in\RR^n$ with $\norm{v}\le\kappa_h$,
\begin{align}\label{eq:breg-ball}
d_h+v\in\Kpolar.
\end{align}
\end{theorem}

\begin{proof}
The first claim is just
$\scal{d_h}{u^i}=-\scal{\nabla h(c_h)}{u^i}=-\kappa_h$
from~\cref{eq:breg-key}. For~\cref{eq:breg-ball}, given
$\norm{v}\le\kappa_h$ and $u^i$ unit,
\begin{align}
\scal{d_h+v}{u^i}=-\kappa_h+\scal{v}{u^i}\le-\kappa_h+\norm{v}\norm{u^i}=-\kappa_h+\norm{v}\le 0,
\end{align}
so $\scal{d_h+v}{u^i}\le 0$ for every $i$, which is precisely
$d_h+v\in\Kpolar$.
\end{proof}

The standard Euclidean case $h(x)=\tfrac12\norm{x}^2$ recovers
\cref{thm:nonpoly} for polyhedral~$\KK$: $\nabla h(x)=x$ gives
$c_h=\proj_{\aff(B_\KK)}(0)=-d$, so $d_h=-\nabla h(c_h)=-c_h=d$,
and $\kappa_h=\scal{c_h}{c_h}=\norm{d}^2$. \Cref{eq:breg-ball}
becomes $\norm{v}\le\norm{d}^2\Rightarrow d+v\in\Kpolar$, which
is~\cite[Thm.~2.6]{BehlingEtAl_CDC}. A simple non-Euclidean instance is the family
$h_p(x)=\tfrac1p\norm{x}^p$ with $p\ge 2$, for which
$\nabla h_p(x)=\norm{x}^{p-2}x$, $\nabla h_p(0)=0$, and
$\dom h_p=\RR^n$. Minimizing $D_{h_p}(x,0)=h_p(x)$ over
$\aff(B_\KK)$ is equivalent to minimizing $\norm{x}$, so
$c_{h_p}=-d$ (the Euclidean projection target), giving
$\nabla h_p(c_{h_p})=\norm{d}^{p-2}(-d)$,
$d_{h_p}=\norm{d}^{p-2}d$, and
$\kappa_{h_p}=\scal{\nabla h_p(c_{h_p})}{c_{h_p}}=\norm{d}^p$.
\Cref{thm:bregman-ball} then states
\begin{align}\label{eq:hp-ball}
\norm{v}\le \norm{d}^p\implies \norm{d}^{p-2}d+v\in\Kpolar,
\end{align}
which for $p=2$ reduces to the Euclidean estimate and for $p>2$
gives a Bregman direction $d_{h_p}$ at distance $\norm{d}^{p-1}$
from the origin, together with an inscribed ball of radius
$\norm{d}^p$.

The $h_p$ family is essentially a re-scaling of the Euclidean
construction: the Bregman direction is parallel to $d$ and the
only new content is the value of the radius. A genuinely different
example, in which $d_h$ is no longer a scalar multiple of $d$, is
the Mahalanobis quadratic.

\begin{example}[Mahalanobis-Bregman extension]\label{ex:maha}
Let $A\in\Smat^n_{++}$ be a positive-definite matrix and set
$h(x):=\tfrac12 x^\top A x$. Then $\nabla h(x)=Ax$, $\nabla h(0)=0$,
and the Bregman divergence is the squared Mahalanobis distance
$D_h(x,y)=\tfrac12(x-y)^\top A(x-y)$. The Bregman projection of $0$
onto $\aff(B_\KK)$ is the unique $c_h\in\aff(B_\KK)$ with
$Ac_h\perp\aff(B_\KK)-c_h$, that is, the $A$-orthogonal projection
of the origin. Then
\begin{align}\label{eq:maha-d}
d_h=-Ac_h,\qquad \kappa_h=c_h^\top A c_h.
\end{align}
For $\KK=\RR^2_+$ and $A=\mathrm{diag}(a_1,a_2)$, Lagrange
optimality on $\aff(B_\KK)=\{x_1+x_2=1\}$ gives
$c_h=\bigl(a_2/(a_1+a_2),\,a_1/(a_1+a_2)\bigr)$, hence
\begin{align}
d_h=-\frac{a_1a_2}{a_1+a_2}(1,1),\qquad
\kappa_h=\frac{a_1a_2}{a_1+a_2},
\end{align}
which reduces to $d_h=d=-(1/2)(1,1)$ and $\kappa_h=1/2$ when
$a_1=a_2$. For $a_1=2,a_2=1$ the Bregman direction
$d_h=-(2/3)(1,1)$ has Euclidean norm $\tfrac{2\sqrt 2}{3}$, larger
than $\norm{d}=\tfrac{\sqrt 2}{2}$, and the inscribed-ball radius
$\kappa_h=2/3$ is larger than the Euclidean radius
$\norm{d}^2=1/2$. The Mahalanobis weighting tilts the construction
toward the lighter coordinate and trades a longer direction for a
wider feasible neighborhood. The bound $\norm{v}\le\kappa_h$ in
\cref{thm:bregman-ball} is still expressed in the {\em Euclidean}
norm, since the Cauchy--Schwarz step that
proves~\cref{eq:breg-ball} pairs the dual point $\nabla h(c_h)$
against $u^i$ in the standard inner product.
\end{example}

\begin{remark}[non-polyhedral version]\label{rem:bregman-nonpoly}
The proof of \cref{lem:bregman-key} uses only first-order
optimality on an affine set and the orthogonality
$\scal{\nabla h(c_h)}{x-c_h}=0$ on $\aff(B_\KK)$. The same
argument applies in the non-polyhedral setting of
\Cref{sec:nonpoly}: if $\KK$ is a closed convex pointed cone with
$\closu\aff(E_\KK)\subset\inte\dom h$ and
$0\notin\closu\aff(E_\KK)$, then
$c_h:=\argmin_{x\in\closu\aff(E_\KK)}D_h(x,0)$ satisfies
$\scal{\nabla h(c_h)}{u}=\kappa_h$ uniformly for $u\in E_\KK$,
and \cref{thm:bregman-ball} extends to give
$d_h+v\in\Kpolar$ for $\norm{v}\le\kappa_h$.
\end{remark}

\subsection{A mirror-descent application}\label{sec:mirror}

The dual identity~\cref{eq:breg-key} fits the standard
mirror-descent template, in which one updates a dual gradient and
returns to the primal through $\nabla h^*$ (see~\cite{BC:book}
for the Hilbert-space framework). Suppose
$f\colon\RR^n\to\RR$ is convex differentiable and we wish to take
a feasibility-corrected descent step at $\bar x\in\Omega=\{g\le 0\}$
in the geometry induced by $h$, with the active-cone setup
of~\cref{prop:step}. The natural Bregman counterpart
of~\cref{eq:fcpg} is
\begin{align}\label{eq:bcfg}
\bar x_+:=\nabla h^*\bigl(\nabla h(\bar x)+\eta\bigl(d_h-\sigma\nabla f(\bar x)\bigr)\bigr),\qquad
\sigma\in\bigl(0,\sigma^{\star}_h(-\nabla f(\bar x))\bigr),\ \eta>0,
\end{align}
where the Bregman step length is the direct analogue
of~\cref{eq:sigmastar},
\begin{align}\label{eq:sigmastar-h}
\sigma^{\star}_h(w):=
\begin{cases}
+\infty, & \max_{j\in J(\bar x)}\scal{w}{u^j}\le 0,\\[1mm]
\dfrac{\kappa_h}{\max_{j\in J(\bar x)}\scal{w}{u^j}}, & \text{otherwise.}
\end{cases}
\end{align}
The same Cauchy--Schwarz argument that delivered \cref{prop:step}
gives that $d_h+\sigma w$ remains in $\inte\Kpolar$ for every
$\sigma\in[0,\sigma^{\star}_h(w))$. The Mahalanobis case
$h(x)=\tfrac12 x^\top A x$ is a particularly transparent
instance: $\nabla h^*=A^{-1}$, the dual update is the linear
shift $Ax\mapsto Ax+\eta(d_h-\sigma\nabla f(\bar x))$, and the
primal update is the constant-time linear solve. The Mahalanobis
weighting $A$ then plays the role of a preconditioner adapted to
the constraint geometry through $c_h$.

\begin{remark}[comparison with Ouyang--Wang Bregman circumcenters]\label{rem:ouyang}
The Bregman circumcenter studied by Ouyang and
Wang~\cite{Ouyang:2021,Ouyang:2023} is, for a finite point set
$X=\{x^1,\ldots,x^m\}\subset\inte\dom h$, a minimizer of
$\max_{i}D_h(c,x^i)$ over $c$ (the precise variational form
depends on which symmetrization of $D_h$ is chosen), regarded as
a Bregman analogue of the Euclidean equidistant point. Our
construction $c_h=\argmin_{x\in\aff(B_\KK)}D_h(x,0)$ instead
computes the Bregman projection of a fixed reference point (the
origin) onto the affine hull of the conic base. The two notions
of ``Bregman circumcenter'' play different roles. Ouyang--Wang
centre an iterate among several auxiliary points, the way the
original Behling--Bello-Cruz--Santos CRM scheme~\cite{BBS:DR}
centres an iterate among reflections; our $c_h$ produces a single
feasibility certificate $d_h\in\Kpolar$ together with a margin
$\kappa_h$, the way the polyhedral $d=-\circum(B_\KK)$ produces
the Euclidean estimate~\cref{eq:behball}. The two coincide in the
Euclidean case $h(x)=\tfrac12\norm{x}^2$ when the data set $X$ is
taken to be the generators $\{u^1,\ldots,u^p\}$, since the
Euclidean circumcenter of a set of unit vectors is the orthogonal
projection of the origin onto their affine hull. Outside this
case the two notions diverge. The Mahalanobis example above gives
$d_h=-(a_1a_2/(a_1+a_2))(1,1)$ from the projection construction.
The Bregman equidistance locus
$\{c:D_h(c,e_1)=D_h(c,e_2)\}$ for the same $h$ reduces to
$a_1(c_1-1)^2+a_2 c_2^2=a_1 c_1^2+a_2(c_2-1)^2$, that is,
$2a_1 c_1-2a_2 c_2=a_1-a_2$, a line in $\RR^2$ that is parallel
to $\aff(B_\KK)=\{c_1+c_2=1\}$ only when $a_1=a_2$. An
Ouyang--Wang Bregman circumcenter of $\{e_1,e_2\}$, which sits on
this line together with an additional optimality condition
specifying which point of the locus is selected, is therefore
generically not the $A$-projection of the origin onto
$\aff(B_\KK)$, and the two constructions return distinct points.
The two notions answer different questions: equidistance among
finitely many points versus single-projection feasibility
certification.
\end{remark}

\section{Concluding remarks}\label{sec:outlook}

The results of this note all derive from the orthogonality
identity $\scal{d}{u^i}=-\norm{d}^2$ of~\cref{eq:keyident}. From
it follow, in the polyhedral case, the exact admissible polyhedron
and directional-depth formula of \Cref{sec:exact}, the
inverse-Gram-matrix closed form with spectral bounds of
\Cref{sec:gram}, and the aperture identity $\norm{d}=\cos\theta$
of \Cref{sec:aperture}. The non-polyhedral extension of
\Cref{sec:nonpoly} delivers the strongest statement of the paper:
the polar description
$\mathcal{P}_\KK=\bigcap_{u\in E_\KK}\{\scal{v}{u}\le\norm{d}^2\}$
together with the contact-set identification
$\norm{d}^2\,\closu E_\KK$. The Jordan-frame value
$\norm{d}_{\!J}^2=1/r$ and the parallel-resistance formula
$1/\norm{d}^2=\sum_\ell 1/\norm{d_\ell}^2$ position the symmetric
and product cases within a unified picture, while the doubly
nonnegative cone and the obstruction at $p$-cones with $p\ne 2$
test the reach of the hypothesis beyond the symmetric setting.
The algorithmic content of \Cref{sec:step,sec:piecewise,sec:concrete}
reduces to the closed-form step-length
oracles~\cref{eq:linfsigma,eq:socsigma}, applicable in both the
polyhedral and the smooth non-polyhedral regimes. The Bregman
extension of \Cref{sec:bregman} reproduces the inscribed-ball
estimate in dual coordinates and recovers
\cite[Thm.~2.6]{BehlingEtAl_CDC} in the Euclidean case.

We close with three open questions: 

The first concerns intersections $\KK=\bigcap_{\ell=1}^L \KK_\ell$
of finitely generated cones. Writing $d_\ell$ for the
circumcentric direction of $\KK_\ell$, one may ask whether some
convex combination $d^{\,\star}=\sum_{\ell=1}^L\alpha_\ell d_\ell$
is an interior feasible direction for $\Kpolar$, and what
inscribed-ball radius around $d^{\,\star}$ one can guarantee. The
geometric coupling between blocks should enter through the
Friedrichs angles between the affine hulls
$\aff(B_{\KK_\ell})$, in contrast with the direct-product case of
\cref{prop:dirsum}, where the parallel-resistance formula
$1/\norm{d}^2=\sum_\ell 1/\norm{d_\ell}^2$ is exact because the
affine hulls are mutually orthogonal. Already the case $L=2$
raises a sharp question: is $\alpha_1=\alpha_2=\tfrac12$
asymptotically optimal as the angle between the two affine hulls
shrinks to zero? An affirmative answer would give a building-block
strategy for problems with naturally decomposable constraint
sets, including conic relaxations and dual-decomposition
formulations, and would connect to the block-wise
circumcentered-reflection scheme of Behling, Bello-Cruz, and
Santos~\cite{BBS:block} and to the recent parallel CRM and
polyhedral-projection methods~\cite{Barros:PPP2025}. 

The second concerns the algorithmic exploitation of
\cref{prop:step,thm:sharp} within a complete first-order method.
The update~\cref{eq:fcpg} is suited to problems where the
active-cone structure changes between iterates, as in piecewise
smooth constraint systems and in interior-point-like schemes that
switch active sets, and \cref{thm:sharp} pins down the worst-case
directions $w\in\closu E_\KK$ along which the inscribed-ball
margin is tight. The two concrete problem classes of
\Cref{sec:concrete} provide natural test beds: the
$L_\infty$-ball case~\cref{eq:linfprob} admits standard benchmarks
(LASSO-type and robust-regression formulations) on which the
closed-form oracle~\cref{eq:linfsigma} can be tested directly,
and the SOCP case~\cref{eq:socp} admits the CBLIB and MOSEK
problem libraries. A convergence analysis under standard
relative-Lipschitz or Polyak--{\L}ojasiewicz hypotheses, with
rates that exploit the angular bound of \cref{cor:angular} when
$w$ is close to the axis $-d/\norm{d}$, would tie the geometry
developed here to the quantitative first-order theory of
Necoara, Nesterov, and Glineur~\cite{NNG:2019} and the
convex-feasibility framework of~\cite{NPR:2019}. The Bregman
feasibility-corrected step \cref{eq:bcfg} provides the natural
mirror-descent variant of the same scheme. 

The third question concerns the geometric appeal of the
hypothesis $0\notin\closu\aff(E_\KK).$ \Cref{prop:pcone} shows
that the hypothesis is essentially a planarity condition on the
extremal section, which fails for $p$-cones with $p\ne 2$ but can
be repaired by working with a different compact base. Identifying
the broadest class of cones for which a self-dual projection
target survives, homogeneous cones with rational characteristic
exponent are a natural candidate, would extend the picture begun
here.

\bibliographystyle{siamplain}

\end{document}